\documentclass[12pt,a4paper]{amsart}

\usepackage{amsmath,amsthm,amssymb,latexsym,a4wide,tikz,multicol,tikz-cd}
\usepackage{arydshln,multirow}
\usepackage{tikz-qtree,tikz-qtree-compat}
\usepackage{mathtools,stmaryrd}
\tikzset{font=\small}
\usepackage{enumerate}
\usetikzlibrary{matrix,arrows}
\usetikzlibrary{positioning}
\usetikzlibrary{cd}
\usepackage{cite}
\usepackage[colorlinks]{hyperref}
\usepackage[margin=10pt,font=small,labelfont=bf, labelsep=period]{caption}
\usepackage{mathrsfs}
\usepackage{geometry} 
\usepackage[active]{srcltx}

\newtheorem{theorem}{Theorem} [section]
\newtheorem{lemma}[theorem]{Lemma}
\newtheorem{corollary}[theorem]{Corollary}
\newtheorem{proposition}[theorem]{Proposition}
\theoremstyle{definition}
\newtheorem{definition}[theorem]{Definition}
\newtheorem{remark}[theorem]{Remark}
\newtheorem{example}[theorem]{Example}

\numberwithin{equation}{section}
\def\mC{\mathcal{C}}

\def\mD{\mathcal{D}}
\DeclareMathOperator{\unit}{\mathop{\boldsymbol{u}}}
\DeclareMathOperator{\dom}{\mathop{\boldsymbol{d}}}
\DeclareMathOperator{\ran}{\mathop{\boldsymbol{r}}}
\DeclareMathOperator{\pr}{\mathop{\boldsymbol{m}}}
\def\PT{\mathscr{P\!\!T\!}}
\def\I{\mathscr{I\!}}
\def\T{\mathscr{T\!}}

\def\Spec{\mathrm{Spec}}
\usepackage{enumitem}
\setlist[enumerate]{itemsep=0mm}
\setlist[itemize]{itemsep=0mm}

\let\OLDthebibliography\thebibliography
\renewcommand\thebibliography[1]{
  \OLDthebibliography{#1}
  \setlength{\parskip}{0pt}
  \setlength{\itemsep}{0pt plus 0.3ex}
}

\begin{document}

\title[A topological approach to restriction semigroups]{A topological approach to discrete restriction semigroups 
and their algebras}

\date{}
\author{Ganna Kudryavtseva}
\address{G. Kudryavtseva: University of Ljubljana,
Faculty of Mathematics and Physics, Jadranska ulica 19, SI-1000 Ljubljana, Slovenia / Institute of Mathematics, Physics and Mechanics, Jadranska ulica 19, SI-1000 Ljubljana, Slovenia}
\email{ganna.kudryavtseva\symbol{64}fmf.uni-lj.si}

\thanks{This work was partially supported by the Slovenian Research and Innovation Agency grants P1-0288 and J1-60025.}

\begin{abstract}   
We introduce a general framework, based on \'etale topological categories, for studying discrete restriction semigroups and their algebras. Generalizing Paterson's universal groupoid of an inverse semigroup, we define the universal category ${\mathscr C}(S)$ of a
restriction semigroup $S$ with local units as the category of germs of the spectral action of $S$ on the character space of its projection semilattice. This is an \'etale topological category, meaning that its domain map is a local homeomorphism, while its range map is only required to be continuous. We show that $S$ embeds into the universal Boolean restriction semigroup of compact slices of ${\mathscr C}(S)$ and apply this embedding to establish the following results:

\begin{itemize}
\item a topological version of the ESN-type theorem for restriction semigroups by Gould and Hollings;
\item an extension to restriction semigroups of the Petrich-Reilly structure theorem for $E$-unitary inverse semigroups in terms of partial actions;
\item an isomorphism between the semigroup algebra of a restriction semigroup $S$ with local units and the convolution algebra of
the universal category ${\mathscr C}(S)$, extending the seminal result by Steinberg.
\end{itemize}
The paper is inspired by the work of Cockett and Garner and builds upon the earlier research of the author. It shows that the theory of restriction semigroups can be developed much further than was previously thought, as a natural extension of the  inverse semigroup theory.

\vspace{0.2cm}

\vspace{0.2cm}

\end{abstract}

\maketitle

\tableofcontents 

\section{Introduction}
Restriction semigroups are algebraic counterparts of partial functions, which are ubiquitous throughout mathematics. These are semigroups equipped with a unary operation, $s\mapsto s^*$ that captures the notion of `taking the domain of a partial function'. A prototypical example is the partial transformation semigroup $PT(X)$ of all partial self-maps of a set $X$, where for a partial function $f$, the partial function $f^*$ is the identity map on the domain of $f$. Just as every group can be embedded into a symmetric group due to the Cayley theorem and, more generally, every inverse semigroup can be embedded into a symmetric inverse semigroup (which consists of all partial injective self-maps of a set) due to the Wagner-Preston theorem, every restriction semigroup can be embedded into a partial transformation semigroup. 

Partial symmetries form a special subclass of partial functions and are modelled by inverse semigroups, which have a rich and well-developed theory and are closely related to \'etale groupoids, forming a convenient framework for the study of algebras of dynamical origin. \'Etale groupoids arise naturally as groupoids of germs of inverse semigroup actions on spaces and are particularly useful for modelling invariants studied in topological dynamics, such as homology, full groups and orbit equivalence, see, for example, \cite{GPS95,M12}. Convolution algebras arising from \'etale groupoids, considered  in both analytical and algebraic settings, include many important examples such as Cuntz algebras \cite{C77}, graph and higher-rank graph $C^*$-algebras \cite{KP00, Raeburn05} and their algebraic counterparts \cite{AASM_book, APCHR13}. 

Restriction semigroups, their bi-unary analogues and generalization have been extensively studied over the last few decades, see, e.g., \cite{CDEGZ23, CG12, DKK21, F77, FG93, GG99, GH09, GH09a, JS01, JS03, Jones16, K15, Kud19, Kud25, KL23, KL17, L86, S16, S17, S18, Stokes17, Stokes22, Sz14} and references therein. For a historical note, see the survey \cite{H09}. 
Motivated by the significant role of partial functions in computer science\footnote{For an overview of the developments in category theory which let to introducing restriction categories and corresponding references, we refer the reader to the introduction of \cite{CL02}.}, multi-object generalizations of restriction monoids,  termed restriction categories, have received much attention from the category theoretic perspective, see, e.g., \cite{CG21,CGH12, CL02, CL03, CL07, CM09, CL24, HL21}. Although the algebraic theory of restriction semigroups has been extensively developed, their connection with topological categories has not been systematically studied, and has been explored only for the classes of Boolean restriction and birestriction semigroups \cite{CG21, Kud25,KL17}. The work of Lawson and the author \cite{KL17} made the first step, extending the non-commutative Stone duality between Boolean inverse semigroups and ample topological groupoids \cite{L10, L12, LL13, R07} to the duality between Boolean birestriction semigroups and biample topological categories\footnote{In \cite{KL17}, birestriction semigroups are referred to as two-sided restriction semigroups, and biample categories as Boolean \'etale categories.}. These results were recently applied by Machado and de Castro, who defined and studied non self-adjoint operator algebras \cite{MC24}, which generalize groupoid $C^*$-algebras. 

The pioneering work of Cockett and Garner \cite{CG21} and its subsequent extension by the author \cite{Kud25} have shown that representation in topological categories is possible even for Boolean restriction semigroups, despite their lack of the range operation. The present article in turn initiates a systematic development of topological methods for the study of arbitrary restriction semigroups and their algebras. The cornerstone of our theory is the universal category ${\mathscr C}(S)$ of a restriction semigroup $S$ with local units, which we define as a generalization of the notion of the universal groupoid of an inverse semigroup \cite{Paterson, Exel08, St10}. The following theorem provides a brief summary of our results on representing $S$ inside ${\mathscr C}(S)$ and on the restriction semigroup of compact slices of ${\mathscr C}(S)$, see Proposition \ref{prop:embedding}, Remark~\ref{rem:identity} and Theorem~\ref{th:univ_main}.

\begin{theorem}\label{th:intro1}\mbox{}
\begin{enumerate}
\item Let $S$ be a restriction semigroup with local units. Then $S$ acts on the spect\-rum $\widehat{P(S)}$ of its projection semilattice $P(S)$. This leads to the category of germs ${\mathscr C}(S)$, which is an ample topological category and is called the universal category of $S$. The restriction semigroup $S$ embeds into the Boolean restriction semigroup ${\mathscr C}(S)^a$ of compact slices of ${\mathscr C}(S)$. Furthermore, ${\mathscr C}(S)^a$ has the universal property with respect to non-degenerate morphisms from $S$ to Boolean restriction semigroups.
\item Let $S$ be an arbitrary restriction semigroup. Then $S$ embeds into the Boolean restriction semigroup ${\mathscr C}(S^1)^a$. Consequently, any restriction semigroup is isomorphic as a restriction semigroup of compact slices of an ample category.
\end{enumerate}
\end{theorem}
Our categorical representation of a restriction semigroup may be viewed as, roughly speaking, the `action analogue' of the classical universal representation of a semilattice in a ring of sets. Because the range map of an ample category is not necessarily open,  a compact slice may have a range which is not open and thus algebraically undetectable. Therefore, unlike in the case of inverse semigroups, working with the action of $S$ on the spectrum of $P(S)$ rather than on $P(S)$ itself is, in the case of restriction semigroups, indispensable. The effectiveness of our approach is demonstrated by several applications. We anticipate that more applications will emerge in the future, related to both restriction semigroups themselves, and to the theory of their algebras and operator algebras.

The first application concerns extensions of the ESN (Ehresmann-Schein-Nam\-boo\-ripad) theorem, usually called ESN-type theorems, which is an active and important area of research, see, e.g., \cite{EAM24}. The ESN theorem \cite{Lawson_book} states that the category of inverse semigroups is isomorphic to the category of inductive groupoids. The absence of the range operation for restriction semigroups motivated Gould and Hollings to define one-sided versions of categories, termed constellations  \cite{GH09}, which led to an  ESN-type theorem, see also \cite{GS17, GS22, JS01, JS03}.  Since the paper \cite{GH09} appeared, it has been widely presumed that  correspondence of restriction semigroups with any kind of category is impossible, or at least unnatural (see \cite[p.271]{GS17} or \cite[p.444]{Stokes17}). However, Theorem \ref{th:intro1} implies that, perhaps quite surprisingly, the category of restriction semigroups is equivalent to the category of inductive constellations of compact slices of ample categories, which also leads to a new proof of the Gould-Hollings theorem. See Section~\ref{s:see_ranges}.

The second application extends McAlister's celebrated result on the structure of $E$-unitary inverse semigroups \cite{McAlister74}, formulated in terms of partial actions \cite{PR79, KL04}, to restriction semigroups. While the existing literature describes the structure of proper restriction semigroups\footnote{Proper restriction semigroups are an appropriate generalization $E$-unitary inverse semigroups.} in terms of actions  \cite{BGG10, CDEGZ23, F77, GG99, GSz13, L86}, the extension of arguably the most powerful variation of the $P$-theorem -- the one representing $E$-unitary inverse semigroups as partial action products -- has not been developed thus far.  
Inspired by the work of Milan and Steinberg \cite{MS14}, we fill this gap by working with partial actions of a monoid acting on the spectrum of a semilattice. This is carried out in Section~\ref{s:proper}. The main results are Theorem~\ref{th:isom_proper}, which shows that the category of germs of an action of a proper restriction semigroup $S$ is topologically isomorphic to the partial transformation category of the corresponding action of the maximum reduced quotient $S/\sigma$ of $S$, and Theorem \ref{th:main1}, which describes the structure of proper restriction semigroups in terms of partial actions and extends the corresponding result for $E$-unitary inverse and proper birestriction semigroups \cite{KL04,PR79,K15}. As a byproduct, we give a new proof of the Petrich-Reilly theorem on the structure of $E$-unitary inverse semigroups \cite{PR79}.

Finally, the third application extends Steinberg's influential result \cite{St10}, which states that the semigroup algebra $KS$ of an inverse semigroup $S$ over a commutative unital ring $K$ is isomorphic to the convolution algebra $K{\mathscr G}(S)$ of the universal groupoid ${\mathscr G}(S)$. This isomorphism provides a direct link between inverse semigroup algebras and the algebras of \'etale groupoids, now widely known as Steinberg algebras. We prove a similar isomorphism result for restriction semigroups with local units: the semigroup algebra $KS$ of such a semigroup is isomorphic to the convolution algebra of the universal category $K{\mathscr G}(S)$, see Theorem \ref{th:isom}. Although our result is formulated similarly to  Steinberg's, our proof is substantially different. The argument in \cite{St10} relies on the underlying groupoid of an inverse semigroup, which has no analogue for restriction semigroups. Our arguments are direct and, when applied to the inverse case, yield a new proof of Steinberg's result that avoids reference to the underlying groupoid.

We conclude the introduction by a brief description of the contents of the paper. Section \ref{s:prelim}  collects the necessary preliminaries for comfortable reading of the paper. In Section \ref{s:brs} we present a brief summary of  notions and results related to ample categories and Boolean restriction semigroups from \cite{Kud25}, which are needed in the sequel. Section~\ref{s:monoid_actions} defines partial actions of monoid on locally compact and Hausdorff spaces and their \'etale categories. Section \ref{s:germs} introduces our central technique. We define an action of a restriction semigroup on a locally compact Hausdorff space and its category of germs. We pay particular attention to the special cases and variations, by looking at actions of range and birestriction semigroups, and show how all of these generalize actions of inverse semigroups. In Section~\ref{s:bool}, we introduce the universal category and the universal Booleanization of a restriction semigroup with local units. Applying the author's results from \cite{Kud25}, we establish the universal property of the universal Booleanization in Theorem \ref{th:univ_bool}. Sections \ref{s:see_ranges}, \ref{s:proper} and~\ref{s:algebras} are devoted to applications of Theorem \ref{th:intro1}. In Section~\ref{s:see_ranges} we prove the equivalence of the categories of inductive constellations and inductive constellations of compact slices of ample categories, see Theorem~\ref{th:main15}. Section~\ref{s:proper} concerns actions and structure of proper restriction semigroups. Section \ref{s:algebras} culminates at Theorem~\ref{th:isom} that extends Steinberg's isomorphism theorem \cite[Theorem 6.3]{St10}. Finally, Section \ref{s:last} outlines several directions for future research suggested by this paper.

\section{Preliminaries}\label{s:prelim}
Throughout the paper, by an {\em algebra} we mean an ordered pair $(A; F)$, where $A$ is a non-empty set and $F$ a collection of finitary operations on $A$, see \cite[Definition 1.1]{McKMcNT87}. If $(A; f_1, \dots, f_n)$ is an algebra, its {\em signature} is the ordered $n$-tuple $(f_1, \dots, f_n)$, and its {\em type} is the ordered $n$-tuple $(\rho(f_1),\dots, \rho(f_n))$, where $\rho(f_i)$ is the arity of the operation $f_i$ for each $i\in \{1,\dots, n\}$. If $(A; f_1, \dots, f_n)$ is an algebra of type $(\rho(f_1),\dots, \rho(f_n))$, we say that it is a $(\rho(f_1),\dots, \rho(f_n))$-algebra. Whenever this does not cause ambiguity,  $(A; f_1, \dots, f_n)$ is denoted simply by~$A$. Morphisms and subalgebras are assumed to respect the given signatures. To emphasise this, we sometimes refer to such a subalgebra or such a morphism as an $(f_1,\dots, f_n)$-subalgebra or a $(f_1, \dots, f_n)$-morphism.

\subsection{Restriction, birestriction and range semigroups} 
The main objects of study in this paper are restriction semigroups. These are semigroups equipped with a unary operation, denoted by $^*$, which serves as an abstract analogue of taking the identity map on the domain of a partial function. We begin, however, by introducing the slightly more general concepts of Ehresmann semigroups and their bi-unary analogues. For a detailed background, we refer the reader to the survey \cite{G10}. 

\begin{definition} (Ehresmann, coEhresmann and biEhresmann semigroups) An {\em Ehresmann semigroup} is an algebra $(S; \cdot \,, ^*)$, where $(S;\cdot)$ is a semigroup and $^*$ is a unary operation on $S$ such that:
\begin{equation}\label{eq:axioms_star}
xx^*=x, \qquad \quad x^*y^*=y^*x^* = (x^*y^*)^*, \qquad \quad (xy)^*=(x^*y)^*.
\end{equation}

A {\em coEhresmann semigroup} is defined dually as an algebra $(S; \cdot \,, ^+)$, where $(S;\cdot)$ is a semigroup and $^+$ is a unary operation on $S$ such that:
\begin{equation}\label{eq:axioms_plus}
x^+x=x, \qquad \quad x^+y^+=y^+x^+=(x^+y^+)^+, \qquad \quad (xy)^+=(xy^+)^+.
\end{equation}

A {\em biEhresmann semigroup}\footnote{BiEhresmann semigroups appear in the literature as {\em two-sided Ehresmann semigroups} as {\em Ehresmann semigroups}. Multi-object generalizations of biEhresmann monoids have been appeared in the category theory literature as {\em bisupport categories.}} is an algebra $(S; \cdot\, , ^*, ^+)$, where $(S;\cdot\, , ^*)$ is an Ehresmann semigroup, $(S;\cdot\, , ^+)$ is a  coEhresmann semigroup and the operations $^*$ and $^+$ satisfy the identities
$(x^+)^*=x^+$  and $(x^*)^+=x^*$.
\end{definition}

The identities $x^* = x^*x^*$ and $(x^*)^* = x^*$  follow from \eqref{eq:axioms_star}, and their dual identities 
 $x^+ = x^+x^+$ and $(x^+)^+ = x^+$
 follow from  \eqref{eq:axioms_plus}.  The operations $^*$ and $^+$ are called  the {\em domain operation} and the  {\em range operation}, respectively\footnote{These operations are also termed the {\em support operation} and the {\em cosupport operation}.}. 

\begin{remark}
In the category theory literature multi-object generalizations of Ehresmann, coEhresmann and biEhresmann monoids are known as {\em support}, {\em cosupport} and {\em bisupport} categories (see \cite{CGH12, HL21}), and the elements $s^*$ and $s^+$ are often denoted by $\bar{s}$ and $\hat{s}$, respectively.   (Here, as usual, a monoid is a semigroup with an identity element $1$, thus an Ehresmann monoid is an Ehresmann semigroup with an identity element.) 
\end{remark}

Recall that a {\em semilattice} is a poset $E$ such that for any elements $e,f\in E$ their greatest lower bound $e\wedge f$ exists in $E$. If $E$ is a semilattice, then $(E; \,\wedge)$ is a commutative idempotent semigroup. Conversely, if $(E;\,\cdot)$ is a commutative idempotent semigroup, define 
$e\leq f$  if and only if  $e=ef=fe$.
Then $\leq$ is a partial order and $(E,\leq)$ is a semilattice with $e\wedge f = e\cdot f$. Denoting the product by juxtaposition, we write $e\wedge f  = ef$. It follows that semilattices can be identified with idempotent and commutative semigroups. Saying that $E$ is a semilattice, we always assume that $E$ is an idempotent and commutative semigroup, equipped with the partial order $\leq$ given above. Saying that $E$ is a {\em meet-semilattice}, we emphasise that we are considering it with respect to the meet-operation $\wedge$ (which is particularly important to state for structures possessing both a meet and a join operation).

If $S$ is an Ehresmann semigroup, then the set $P(S) = \{s^*\colon s\in S\}$
is closed with respect to the multiplication and is a semilattice. It is called 
the {\em projection semilattice} of $S$ and its elements are called {\em projections}.\footnote{In the category theory literature projections are called {\em restriction idempotents.}} 
It follows from \eqref{eq:axioms_star} that $P(S)$ can be equivalently defined as the set of all $s\in S$ satisfying $s^*=s$. If $S$ is biEhresmann semigroup, we have $P(S)=\{s^+\colon s\in S\} = \{s^*\colon s\in S\}$.

Restriction (resp. corestriction or birestriction) semigroups form a subclass of Ehresmann (resp. coEhresmann or biEhresmann) semigroups and are defined as follows. 

\begin{definition} (Restriction, corestriction and birestriction semigroups)
A {\em restriction semigroup}\footnote{Restriction and corestriction semigroups appear in the literature  as {\em right  restriction semigroups} (or {\em weakly right ample semigroups}) and {\em left restriction semigroups} (or {\em weakly left ample semigroups}), respectively. Our terminology follows \cite{Kud25} and is consistent with the category theory literature, where multi-object versions of restriction monoids are called {\em restriction categories.}} is an Ehresmann semigroup $(S; \cdot \,, ^*)$ which satisfies the identity
\begin{equation}\label{eq:axioms_star_restr}
x^*y = y(xy)^*.
\end{equation}
A {\em corestriction semigroup} is a coEhresmann semigroup $(S; \cdot \,, ^+)$ which satisfies the dual identity
$xy^+ = (xy)^+x$.
A {\em birestriction semigroup}\footnote{Just as with biEhresmann semigroups, birestriction semigroups appear in the literature  as {\em two-sided restriction semigroups} or as  {\em restriction semigroups}.} is a bi\-Ehres\-mann semigroup 
such that both \eqref{eq:axioms_star_restr} and its dual identity hold.
\end{definition}

The following identities easily follow from the definitions.
\begin{enumerate}
\item[1.] Let $S$ be an Ehresmann semigroup. Then
\begin{equation}\label{eq:rule1r}
(se)^* = s^*e \quad \text{ for all } \,\,  s\in S \,\, \text{ and } \,\, e\in P(S).
\end{equation}
\item[2.] Let $S$ be a  birestriction semigroup. Then
\begin{equation}\label{eq:ample_r}
es = s(es)^* \,\, \text{ and } \,\, se = (se)^+s \quad \text{ for all } \,\,  s\in S \,\, \text{ and } \,\, e\in P(S).  
\end{equation}
\end{enumerate}

Recall that a semigroup $S$ is called an {\em inverse semigroup} if for each $a\in S$, there exists a unique $b\in S$ such that $aba=a$ and $bab=b$. The element $b$ is the {\em inverse} of $a$ and is denoted by $a^{-1}$. When $(S;\,\cdot)$ is an inverse semigroup, defining $a^+=aa^{-1}$ and $a^*=a^{-1}a$ for each $a\in S$ makes $(S; \cdot,\, ^*, ^+)$ a birestriction semigroup (and thus also a biEhresmann semigroup). If $S$ is an inverse semigroup then the semilattice $P(S)$ coincides with the semilattice $E(S)$ of idempotents of $S$.

Range semigroups were introduced in \cite[Definition 3.8]{Kud25} and are inspired by range categories of \cite{CGH12}.

\begin{definition} (Range semigroups)
By a {\em range semigroup} we mean a biEhresmann semigroup $(S;\cdot \,, ^*, ^+)$ such that $(S; \cdot \,, ^*)$ is a restriction semigroup.
\end{definition}

Before proceeding, we provide prototypical examples of the unary and bi-unary semigroups thus defined.

\begin{example} \label{ex:1} 
\begin{enumerate}
\item The monoid $R(X)$\footnote{In \cite{Kud25}, the monoid $R(X)$ was denoted by ${\mathcal{B}}(X)$, but in this paper the latter notation has a different meaning.} of all binary relations on a set $X$, together with the operations $^*$ and $^+$ given by
\begin{align*}
\rho^* & = \{(x,x)\colon (y,x)\in \rho \text{ for some } y\in X\} \quad \text{ and } \\
\rho^+ & = \{(x,x)\colon (x,y)\in \rho \text{ for some } y\in X\},
\end{align*}
is biEhresmann. If $|X|\geq 2$, it is neither restriction nor corestriction.
\item The $(\cdot \,, ^* ,^+)$-subalgebra $PT(X)$ of $R(X)$ consists of all partial self-maps of $X$ (that is, maps $\varphi \colon {\mathrm{dom}}(\varphi)\to X$ where ${\mathrm{dom}}(\varphi)\subseteq X$ is the {\em domain} of $\varphi$) is a restriction monoid with respect to the signature $(\cdot,\, ^*)$, and a range monoid with respect to the signature $(\cdot \,, ^*, ^+)$. If $|X|\geq 2$ it is not corestriction. 
\item The $(\cdot \,, ^*,^+)$-subalgebra $I(X)$ of $PT(X)$ consisting of all partial bijections of $X$ (that is, of all $\varphi\in PT(X)$ which are injective) is a birestriction monoid. Defining $\varphi^{-1}$ to be the inverse partial bijection of $\varphi$, $I(X)$ is an inverse monoid.
\end{enumerate} 
\end{example}
The range monoid $PT(X)$ is known as the {\em full transformation monoid} on $X$, and the inverse monoid $I(X)$ is called the {\em symmetric inverse monoid} on $X$.

\begin{definition} (Local units) \label{def:lu} We say that a restriction semigroup $S$ has {\em local units} \cite[Definition 5.4]{Kud25} if for every $s\in S$ there are $e,f\in P(S)$ such that $es=sf=s$. 
\end{definition} 
Since $ss^* =s$, $S$ has local units if and only if for every $s\in S$ there is $e\in P(S)$ such that $s=es$. Due to the identity $s^+s=s$, any range semigroup has local units, as clearly does any restriction monoid. If $S$ does not have an identity element then $S^1 = S\cup \{1\}$ where $1\not\in S$ is a restriction monoid if one sets $1^* = 1$.

\subsection{Proper restriction semigroups} \label{subs:proper} Let $S$ be a restriction semigroup. The {\em natural partial order} on $S$ is defined by $s\leq t$ if and only if $s=ts^*$, which is equivalent to $s=te$ for some $e\in P(S)$. The restriction of the partial order $\leq$ to $P(S)$ coincides with the partial order on the semilattice $P(S)$.
The following easy properties (see \cite[Lemmas 3.6 and 3.7]{Kud25}) of the natural partial order will be used throughout the paper, possible without explicit mention.

\begin{lemma}\label{lem:1} Let $S$ be a restriction semigroup. 
\begin{enumerate}
\item Let $a,b\in S$ and $a\leq b$. Then $a^*\leq b^*$.
\item Let $e,f\in P(S)$, $s\in S$ and $e\leq f$. Then $(es)^* \leq (fs)^*$.
\item  Let $s,t,u\in S$ and $s\leq t$. Then $su \leq tu$ and $us \leq ut$.
\end{enumerate}
\end{lemma}

We say that $S$ is {\em reduced}, if it contains only one projection. A reduced restriction semigroup is necessarily a monoid, with   the identity element $1$ being its unique projection. The {\em minimum reduced congruence} $\sigma$ on $S$ is the minimum congruence that identifies all the projections. The quotient $S/\sigma$ is then the maximum reduced quotient of $S$. Although the following lemma is known, we include its proof here for completeness.

\begin{lemma} \label{lem:s14a} Let $a,b\in S$. Then $a \mathrel{\sigma} b$ if and only if there is $c\in S$ such that $c\leq a,b$.
\end{lemma}

\begin{proof} Define the relation $\rho$ on $S$ by $a\mathrel{\rho} b$ if and only if there is $c\in S$ such that $c\leq a,b$. We show that $\rho = \sigma$.
To show that $\rho$ is transitive, let $a \mathrel{\rho} b$ and $b\mathrel{\rho} c$. Then there are $s,t\in S$ such that $s\leq a,b$ and $t \leq b,c$.
Then $s=bs^*$ and $t=bt^*$. It follows that $bs^*t^*\leq s,t$, so that $bs^*t^* \leq a,c$. Thus $a\mathrel{\rho} c$. Since $\rho$ is trivially reflexive and symmetric, it is an equivalence relation. That $\rho$ respects the operations $\cdot$ and $^*$ follows from Lemma \ref{lem:1}, yielding that $\rho$ is a congruence. Let $e,f\in P(S)$. Since $ef \leq e,f$, it follows that $e \mathrel{\rho} f$. Since $\rho$ identifies all projections and $\sigma$ is the minimum congruence with this property, we have $\sigma\subseteq \rho$. For the reverse inclusion, assume $a\mathrel{\rho} b$. If $c\leq a,b$ then $a = aa^* \mathrel{\sigma} ac^* = c = bc^* \mathrel{\sigma} bb^* = b$. Hence, $\rho \subseteq \sigma$. It follows that $\rho = \sigma$.
\end{proof}

The {\em compatibility relation} $\smile$ on a restriction semigroup $S$ is defined by $s\smile t$ if and only if $st^* = ts^*$. It can be readily shown (or see, e.g., \cite[Lemma 4.15]{Kud25}) that if $s$ and $t$ have a common upper bound, then $s\smile t$. Moreover, if $s\smile t$ and $u\in S$ then $us \smile ut$ and $su\smile tu$, see \cite[Lemma~4.13]{Kud25}.
We prove the following property of the compatibility relation, which will be needed in the sequel.

\begin{lemma} \label{lem:compatibility} Suppose that $s_1, \dots, s_n$, where $n\geq 1$, are pairwise compatible, and let $e=s_1^*\cdots s_n^*$. Then the meet $s_1\wedge \dots \wedge s_n$ with respect to the natural partial order $\leq$ exists and coincides with each of the elements $s_1e$, $\dots$, $s_ne$. 
\end{lemma}

\begin{proof} 
We first prove that if $s\smile t$ and $u\leq s$ then $u\smile t$. Indeed, bearing in mind that $u^*\leq s^*$, we have $tu^* = tu^*s^* = ts^*u^* = st^*u^* = su^*t^* = ut^*$, as required.

We prove the statement by induction on $n$. If $n=1$, the statement is trivially correct. Suppose that $n\geq 2$. By the inductive assumption, we have that the meet $s=s_1\wedge \dots \wedge s_{n-1}$ exists and equals to each of the elements $s_1\tilde{e}$, $\dots$, $s_{n-1}\tilde{e}$ where $\tilde{e}=s_1^*\cdots s_{n-1}^*$. This implies that $s^* = (s_1\tilde{e})^* = s_1^*\tilde{e} = \tilde{e}$. Since $s\leq s_1$ and $s_1\smile s_n$, the above claim implies that $s\smile s_n$. By the definition of the compatibility relation we then have 
 $ss_n^* = s_ns^*$. We prove that this element is the meet of $s$ and $s_n$. We have that $ss_n^* = s_ns^* \leq s, s_n$ is the lower bound of $s$ and $s_n$. Suppose that $t$ is their other lower bound. Then $t=st^* = s_nt^*$, so that $t^* = (st^*)^* = s^*t^* \leq s^*$, and similarly $t^*\leq s_n^*$. It follows that $t^*\leq s^*s_n^*$. Since $t$ and $ss_n^*$ have a common upper bound, $s$, they are compatible. It follows that $t= tt^* = tt^*s^*s_n^* = t(ss_n^*)^* = ss_n^*t^* \leq ss_n^*$. This proves that $ss_n^*$ is the greatest lower bound of $s$ and $s_n$, as required. It follows that the meet $s\wedge s_n = s_1\wedge \dots \wedge s_n$ coincides with $s_ns^*$ which equals to $s_n \tilde{e} = s_n(\tilde{e}s_n^*) = s_ne$, and the statement follows.
\end{proof}

\begin{definition} (Proper restriction semigroup) \label{def:proper} A restriction semigroup $S$ is called {\em proper} \cite{GG99, F77}, if for all $s,t\in S$ such that $s\mathrel{\sigma} t$ and $s^* = t^*$, one has that $s=t$. 
\end{definition}

Proper restriction semigroups generalize $E$-unitary inverse semigroups. Specifically, an inverse semigroup $S$ (viewed as a restriction semigroup with $s^*=s^{-1}s$) is proper if and only if it is $E$-unitary.

The following lemma is known, but we give a proof for completeness.

\begin{lemma} \label{lem:proper} Let $S$ be a restriction semigroup. Then $S$ is proper if and only if the compatibility relation $\smile$ coincides with $\sigma$.
\end{lemma}

\begin{proof}   If $s\smile t$ then $st^* = ts^*$. Since $st^*\leq s$ and  $ts^* \leq t$, Lemma \ref{lem:s14a} implies that $s\mathrel{\sigma} t$.  
Therefore, we only need to show that $S$ is proper if and only if $s\mathrel{\sigma} t$ implies $s\smile t$.

Suppose $S$ is proper and $s\mathrel{\sigma} t$. Since $s^* \mathrel{\sigma} t^* \mathrel{\sigma} 1$ and $\sigma$ is a congruence, we have $st^* \mathrel{\sigma} s$ and $ts^* \mathrel{\sigma} t$, yielding $st^* \mathrel{\sigma} ts^*$.  Since $S$ is proper and $(st^*)^* = (ts^*)^* = s^*t^*$, it follows that $st^* = ts^*$, which means $s\smile t$. Conversely, suppose  $s\mathrel{\sigma} t$ implies that $s\smile t$. Let $u,v\in S$ be such that $u \mathrel{\sigma} v$ and $u^* = v^*$.
The assumption implies $u\smile v$, which means $uv^* = vu^*$. It follows that $u=uu^* = uv^*= vu^* = vv^*= v$. Therefore, $S$ is proper.
\end{proof}

Let now $S$ be a birestriction semigroup. The natural partial order on $S$ is defined, just as on restriction semigroups, by $s\leq t$ if and only if $s=st^*$. It is known that $s\leq t$ holds if and only if $s = s^+t$ (see \cite[Lemma 2.1]{Kud19}\footnote{Beware, that in \cite{Kud19} birestriction semigroups are called two-sided restriction semigroups, and the bicompatibility relation $\asymp$ is called the compatibility relation, denoted $\sim$.}).

Elements $s,t\in S$ are called {\em bicompatible}, denoted $s\asymp t$, if they are compatible ($st^* = ts^*$) and cocompatible ($t^+s = s^+t$). Observe that $s\asymp t$ implies $s\smile t$, but the converse implication does not hold in general (see \cite[Example 4.18]{Kud25}).

\begin{definition} (Proper birestriction semigroups)
A birestriction semigroup $S$ is called {\em proper} if the following conditions hold:
\begin{enumerate}
\item for all $s,t\in S$: if $s^*=t^*$ and $s\mathrel{\sigma} t$ then $s=t$;
\item for all $s,t\in S$: if $s^+=t^+$ and $s\mathrel{\sigma} t$ then $s=t$.
\end{enumerate}
\end{definition}

Proper birestriction semigroups, just as proper restriction semigroups, generalize $E$-unitary inverse semigroups. The following statement can be proved by slightly adapting the proof of Lemma \ref{lem:proper}.

\begin{lemma} \label{lem:proper1} Let $S$ be a birestriction semigroup. Then $S$ is proper if and only if the bicompatibility relation $\asymp$ coincides with $\sigma$.
\end{lemma}

\subsection{Stone duality for generalized Boolean algebras} \label{subs:Stone1}
For the sake of fixing the notation and for completeness, we briefly recall the duality between generalized Boolean algebras and locally compact Stone spaces \cite{Stone37,D64} (for a resent detailed account, see \cite{L23}), which is one of the main tools used in this paper. A {\em generalized Boolean algebra} is a relatively complemented distributive lattice with the bottom element $0$. By ${\mathbb B} = \{0,1\}$ we denote the smallest non-zero generalized Boolean algebra. If $E,F$ are generalized Boolean algebras, a {\em morphism} $\varphi\colon E\to F$ is a map such that:
\begin{enumerate}
\item $\varphi(0) = 0$;
\item $\varphi(e\wedge f) = \varphi(e)\wedge \varphi(f)$ for all $e,f\in E$;
\item $\varphi(e\vee f) = \varphi(e) \vee \varphi(f)$ for all $e,f\in E$.
\end{enumerate}

 It is easy to verify that a morphism $\varphi\colon E\to F$ satisfies $\varphi(e\setminus f) = \varphi(e)\setminus \varphi(f)$ for all $e,f\in E$. A morphism $\varphi\colon E\to F$ between generalized Boolean algebras is called {\em non-dege\-ne\-rate} (or {\em proper}) if for all $f\in F$ there is $e\in E$ satisfying $\varphi(f)\geq e$, which is equivalent to requiring that the ideal of $F$, generated by $\varphi(E)$, coincides with $F$. If $E$ and $F$ are unital, then $\varphi$ is non-degenerate if and only if it satisfies $\varphi(1) = 1$. In what follows  all morphisms between generalized Boolean algebras are assumed to be non-degenerate.

A Hausdorff space $X$ is a {\em locally compact
Stone space} if it has a basis of compact-open sets. A continuous map  between topological spaces is called {\em proper} if the inverse images of compact sets are compact sets.

Let $E$ be a generalized Boolean algebra. A {\em prime character} of $E$ is a non-zero morphism of generalized Boolean algebras $f\colon E\to {\mathbb B}$. By $\Spec(E)$\footnote{In \cite{Kud25}, we denoted $\Spec(E)$ by $\widehat{E}$, but in this paper $\widehat{E}$ will have a different meaning denoting the spectrum (i.e., the space of characters) of a semilattice $E$ (see Remark \ref{rem:dif}).} we denote the set of all prime characters of $E$. This set is endowed with the locally compact Stone space topology with a basis given by the sets 
$$D(e) = \{\varphi\in \Spec(E)\colon \varphi(e)=1\},$$ where $e$ runs through $E$. The topological space $\Spec(E)$ is called the {\em prime spectrum} or the {\em space of prime characters} of the generalized Boolean algebra $E$.

\begin{definition} (Filter) \label{def:filter} A non-empty subset ${\mathrm F}$ of a poset $(P, \leq)$ is called a {\em filter}, if:
\begin{enumerate}
\item ${\mathrm F}$ is {\em downward directed}, that is, if $e,f\in {\mathrm F}$ then there is $g\leq e,f$ such that $g\in {\mathrm F}$;
\item ${\mathrm F}$ is {\em upward closed}, that is, if $e\in {\mathrm F}$ and $f\geq e$ then $f\in {\mathrm F}$.
\end{enumerate}
\end{definition}

In the particular case where $P$ is a semilattice, we note that a subset ${\mathrm{F}} \subseteq P$ is downward directed if and only if for every $e,f\in {\mathrm F}$, their meet $e\wedge f$ is also in ${\mathrm F}$. Consequently, if $E$ is a semilattice, then 
 a non-empty subset of $E$ is a filter if and only if it upward closed and is a subsemilattice of $E$.
 
If $P$ is a poset and $p\in P$, we denote the {\em principal filter generated by} $p$ by 
$$
p^{\uparrow} = \{q\in P\colon q\geq p\}.
$$

\begin{definition} (Prime filter)
A filter ${\mathrm F}$ of a generalized Boolean algebra $E$ is called {\em prime} if $e\vee f\in {\mathrm F}$ for some $e,f\in {\mathrm F}$ implies $e\in {\mathrm F}$ or $f\in {\mathrm F}$.
\end{definition}

Prime characters of a generalized Boolean algebra $E$ are in a bijection with prime filters of $E$ via the assignments $\varphi\mapsto \varphi^{-1}(1)$ if $\varphi$ is a prime character, and ${\mathrm F}\mapsto \chi_{{\mathrm F}}$ if ${\mathrm F}$ is a prime filter, where $\chi_{{\mathrm F}}$ is the {\em characteristic function} of ${\mathrm F}$. Identifying $\Spec(E)$ with the set of prime filters, we have 
$$D(e) = \{{\mathrm F}\in \Spec(E)\colon e\in {\mathrm F}\}.
$$

The assignment $\Spec\colon E\mapsto \Spec(E)$ gives rise to a contravariant functor from the category of generalized Boolean algebras with non-degenerate morphisms to the category of locally compact Stone spaces with proper and continuous maps. It takes a non-degenerate morphism $f\colon E\to F$ of generalized Boolean algebras to the proper and continuous map $f^{-1}\colon \Spec(F)\to \Spec(E)$. Using the bijection with prime characters, this rewrites to $f^{-1}(\chi_{{\mathrm F}}) = \chi_{{\mathrm F}}\circ f$ where ${\mathrm F}$ is a prime filter of $F$.

If $X$ be a locally compact Stone space, by ${\mathcal B}(X)$ we denote the generalized Boolean algebra of all compact-open sets of $X$ with respect to usual operations on sets. The assignment ${\mathcal B}\colon X\mapsto {\mathcal B}(X)$ gives rise to a contravariant functor from the category of locally compact Stone spaces to the category of generalized Boolean algebras. It takes a proper and continuous map $f\colon X\to Y$ between locally compact Stone spaces to the morphism of generalized Boolean algebras $f^{-1}\colon {\mathcal B}(Y)\to {\mathcal B}(X)$.

The following result dates back to \cite{Stone37, D64}, and its modern exposition can be found in \cite[Proposition 3, Theorem 6]{L23}, see also \cite[Theorem 2.1]{Kud25}.

\begin{theorem} \label{th:duality_sem} 
The functors $\Spec$ and ${\mathcal B}$ establish a dual equivalence between the category of generalized Boolean algebras and the category of locally compact Stone spaces. The natural isomorphism $E\to {\mathcal B}(\Spec(E))$ is given by $e\mapsto \{{\mathrm F}\in \Spec{E}\colon e\in {\mathrm F}\}$.
The natural homeomorphism $X\to \Spec({\mathcal B}(X))$ is given by $x\mapsto {\mathrm F}_x$ where $U\in {\mathrm F}_x$ if and only if $x\in U$.
\end{theorem}

\subsection{The spectrum and the universal Booleanization of a semilattice} \label{subs:spectrum_sem}
Let $E$ be a meet-semilattice. We define $\widehat{E}$ to be  the set of all non-zero morphisms of semilattices from  $E$ to ${\mathbb B}$. A map $\varphi\colon E\to {\mathbb B}$ belongs to $\widehat{E}$ if it is non-zero and satisfies:
$$
\varphi(e\wedge f) = \varphi(e)\wedge \varphi(f), \quad \text{ for all } \quad e,f\in E.\footnote{Some authors consider only semilattices with zero, and also require that $\varphi(0)=0$, see, e.g., \cite{Exel21}, but we do not need to require this.}
$$ 
It follows from \cite[Lemma~34.1]{GHbook}  that  $\widehat{E}$ is a closed subset with zero removed of the Stone space ${\mathbb B}^E$. Consequently, with respect to the relative topology inherited from ${\mathbb B}^E$, it is a locally compact Stone space. The elements of $\widehat{E}$ are called {\em characters}.  A basis of the topology on $\widehat{E}$, called the {\em patch topology}, is formed by the sets 
\begin{equation}\label{eq:de}
D_e = \{\varphi \in \widehat{E}\colon \varphi(e)=1\},
\end{equation}
where $e\in E$, along with the sets
\begin{equation}\label{eq:de1}
D_{e; f_1,\dots, f_k} = D_e \setminus (D_{f_1}\cup \dots \cup D_{f_k}),
\end{equation}
where $k\geq 1$, $e, f_1,\dots, f_k\in E$, and $f_1,\dots, f_k\leq e$. The space $\widehat{E}$ is called the {\em spectrum} or the {\em character space} of the semilattice $E$.

\begin{remark} \label{rem:dif} We emphasise that if $E$ is a generalized Boolean algebra, $\widehat{E}$ and $\Spec(E)$ are different spaces. A non-zero map $\varphi\colon E\to {\mathbb{B}}$
belongs to $\widehat{E}$ if $\varphi(e\wedge f) = \varphi(e)\wedge \varphi(f)$ for all $e,f\in E$. In contrast, $\Spec(E)$ consists of maps $\varphi\in \widehat{E}$ that satisfy the additional conditions $\varphi(0)=0$ and
$\varphi(e\vee f) = \varphi(e)\vee \varphi(f)$ for all $e,f\in E$. This shows that $\Spec(E)$ is a closed subset of $\widehat{E}$.
\end{remark}

Characters on a semilattice $E$ are in a bijection with {\em filters} of $E$ (see Definition \ref{def:filter}). If $f\colon E\to {\mathbb{B}}$ is a character, then $f^{-1}(1)$ is a filter, the inverse bijection taking a filter ${\mathrm{F}}$ to the character $\chi_{\mathrm{F}}$. 
In the sequel, elements of $\widehat{E}$ will be viewed interchangeably as filters or as characters, with the intended meaning determined by the context.

Since $\widehat{E}$ is a locally compact Stone space, its compact-open sets, which are precisely the sets given in \eqref{eq:de} and \eqref{eq:de1}, form a generalized Boolean algebra ${\mathcal B}(\widehat{E})$.  In order to formulate the universal property of ${\mathcal B}(\widehat{E})$, we need the following definition, which is taken from \cite{Exel21}. 

\begin{definition} (Non-degenerate morphism) Let $E$ be a meet-semilattice and $B$ a generalized Boolean algebra. We call a morphism of meet-semilattices $\alpha\colon E\to B$ {\em non-degenerate} if
the ideal of $B$, generated by $\alpha(E)$, coincides with $B$.
\end{definition}

\begin{proposition} \label{prop:universal_sem} 
Let $E$ be a semilattice, $B$ a generalized Boolean algebra, and let $\alpha\colon E\to B$ be a non-degenerate morphism of meet-semilattices. Then there exists a morphism of generalized Boolean algebras $\psi \colon {\mathcal B}(\widehat{E}) \to B$ such that $\alpha = \psi \iota$. 
\end{proposition}

\begin{proof} Applying Theorem \ref{th:duality_sem}, it suffices to prove that there is a proper and continuous map $\Psi\colon \Spec(B)\to \widehat{E}$. Let $\varphi \in \Spec(B)$. Then $\varphi\alpha\colon E\to {\mathbb B}$ is a morphism of semilattices. Since $\varphi$ is non-zero, there is $b\in B$ such that $\varphi(b)=1$. Since $\alpha(E)$ generates $B$, we can write $b=b_1\vee \cdots \vee b_k$ where $k\geq 1$ all $b_i$ belong to $\alpha(E)$. Since $\varphi$ is a prime character, $\varphi(b_i)=1$ for some $i\in \{1,\dots, k\}$, so that there is $e\in E$ such that $\varphi\alpha(e)=1$, which shows that $\varphi\alpha$ is a character of $E$. We set $\Psi(\varphi) = \varphi\alpha \in \widehat{E}$. Observe that $\Psi(\varphi) = \varphi\alpha \in D_e$ if and only if $\varphi\alpha(e) = 1 \Leftrightarrow \varphi \in D({\alpha(e)})$. Hence $\Psi^{-1}(D_e) = D({\alpha(e)})$. Similarly, $\Psi^{-1}(D_{e; f_1,\dots, f_k}) = D({\alpha(e)})\setminus (D({\alpha(f_1)})\cup \cdots \cup D({\alpha(f_k)}))$. This shows that $\Psi$ is continuous.
\end{proof}
This preceding proposition motivates the following definition.

\begin{definition} \label{def:univ_b_sem} (Universal Booleanization of a semilattice)
Let $E$ be a semilattice. The generalized Boolean algebra ${\mathcal B}(\widehat{E})$ is called the {\em (universal) Booleanization} of $E$ (or the {\em enveloping generalized Boolean algebra} of $E$) and will be denoted by ${\mathsf B}(E)$. 
 \end{definition}

Since $D_e\cap D_f = D_{e\wedge f}$ for every $e,f\in E$, the map 
$$
\iota\colon E\to {\mathcal B}(\widehat{E}), \qquad e\mapsto D_e
$$
is a morphism of meet-semilattices. Furthermore,  $D_e=D_f$ implies $\chi_{e^{\uparrow}} \in D_f$, which means $f\geq e$. By symmetry, we also have $e\geq f$, so that $e=f$. Thus, $\iota$ is an embedding, which shows that $\iota(E)$ is a subsemilattice of ${\mathcal B}(\widehat{E})$ that is isomorphic to $E$. 
We arrive at the following statement.

\begin{proposition} \label{prop:7o} Let $E$ be a semilattice. Then: 
\begin{enumerate}
\item $\iota(E)$ generates ${\mathcal B}(\widehat{E})$ as a generalized Boolean algebra;
\item the map
$\Psi\colon\Spec({\mathcal B}(\widehat{E}))\to \widehat{E}$, $\varphi\mapsto \varphi\iota$
is a homeomorphism; 
\item for every $e\in E$ we have:
$$
\Psi^{-1}(D_e)  = D(e);
$$
for all $k\geq 1$ and $e,f_1,\dots, f_k\in E$ satisfying $f_1,\dots, f_k\leq e$, we have:
$$
\Psi^{-1}(D_{e; f_1,\dots, f_k}) = D({\alpha(e)})\setminus \left(D({\alpha(f_1)})\cup \cdots \cup D({\alpha(f_k)})\right). 
$$
\end{enumerate}
\end{proposition}

\begin{proof}
(1) The sets $D_e$ and $D_{e;f_1,\dots, f_k}$ are compact-open and form a basis of the topology on $\widehat{E}$. Hence, the generalized Boolean algebra ${\mathcal B}(\widehat{E})$ consists of all finite joins of these sets, so, in view of \eqref{eq:de1}, the sets $D_e$, $e\in E$, generate ${\mathcal B}(\widehat{E})$ as a generalized Boolean algebra. 
 In view of Proposition~\ref{prop:universal_sem}, the non-degenerate morphism $\iota\colon E\to {\mathcal B}(\widehat{E})$ gives rise to the identity isomorphism of the generalized Boolean algebra ${\mathcal B}(\widehat{E})$. The corresponding homeomorphism $\Psi\colon \Spec({\mathcal B}(\widehat{E})) \to \widehat{E}$ from the proof of Proposition \ref{prop:universal_sem} is given by $\varphi\mapsto \varphi\iota$, which establishes (2). Part (3) also follows directly from the same proof.
\end{proof}

\section{Boolean restriction semigroups and ample categories}\label{s:brs}
\subsection{Boolean restriction semigroups}
Boolean restriction semigroups were introduced in \cite[Definition 5.2]{Kud25}. They are a natural generalization of Boolean inverse semigroups, which are the algebraic counterparts of ample topological groupoids \cite{L23, StSz21, W17}. Their category theory analogues, called classical restriction categories, are discussed in \cite{CM09}. Boolean restriction monoids were studied in \cite{G23} (see \cite[Remarks 5.3 and 5.5]{Kud25}). 

Let $S$ be a restriction semigroup. Following \cite{CL07,CL24}, a zero element which is a projection is called a {\em restriction zero}. 
\begin{definition} (Boolean restriction semigroups) \label{def:brs} 
A restriction semigroup $S$ with a restriction zero is called {\em Boolean}  if the following conditions hold:
\begin{enumerate}
\item[(BR1)] For any $s,t\in S$ such that $s\smile t$, the join $s\vee t$ exists in $S$.
\item[(BR2)] $(P(S), \leq$) is a generalized Boolean algebra.
\item[(BR3)]  For any $s,t,u\in S$ such that $s\vee t$ exists, we have $(s\vee t)u = su \vee tu$.
\end{enumerate}
\end{definition}

Condition (BR3) says that in a Boolean restriction semigroup the multiplication distributes over  joins from the right. Furthermore, \cite[Lemma 5.6]{Kud25} ensures also the left distributivity. By a {\em Boolean range semigroup} we mean a range semigroup $(S; \cdot\, ,^*,^+)$ such that $(S; \cdot\, ,^*)$ is a Boolean restriction semigroup.

\begin{definition} (Boolean birestriction semigroups)
 Let $S$ be a  birestriction semigroup with a restriction zero. It is called {\em Boolean} if:
 \begin{enumerate}
\item[(BBR1)] For any two elements $a,b\in S$ such that $a\asymp b$, the join $a\vee b$ exists in $S$.
\item[(BBR2)] $(P(S), \leq$) is a generalized Boolean algebra.
\end{enumerate}
\end{definition}

\subsection{Topological categories}
We start from the precise definition of a category used in this paper. We adopt the approach that a category is a generalization of a monoid which has a partially defined multiplication operation. Adapting the definition of a category given in \cite{M71}, we identify objects with the set of identity morphisms at them. We consider only small categories.
\begin{definition} [Category] \label{def:cat}
Let $\mC$ be a set and $\dom,\ran \colon \mC\to \mC$ be two maps with the same image ${\dom}(\mC) = {\ran}(\mC)$, denoted $\mC^{(0)}$. We refer to the elements of $\mC^{(0)}$ as {\em units} and may sometimes write $\dom,\ran \colon \mC\to \mC^{(0)}$.
The {\em set of composable pairs} is defined as $\mC^{(2)} = \{(x,y)\in \mC\times \mC \colon {\ran}(y)={\dom}(x)\}$. Let, further, $\pr\colon \mC^{(2)}\to \mC$ be a map, and for any $(x,y)\in \mC^{(2)}$, we denote ${\pr}(x,y)$ by $xy$. We say that $\mC$, together with the maps $\dom,\ran$ and  $\pr$, is a (small) category if the following conditions hold:
\begin{enumerate}
\item (domain and range of a unit) If $x\in \mC^{(0)}$ then ${\dom}(x) = {\ran}(x) = x$.
\item (domain of a product) If $(x,y)\in \mC^{(2)}$ then ${\dom}(xy)  = {\dom}(y)$.
\item (range of a product) If $(x,y)\in \mC^{(2)}$ then ${\ran}(xy)  = {\ran}(x)$.
\item (associativity) If $(x,y), (y,z) \in \mC^{(2)}$ then $(xy)z  = x(yz)$.
\item (left and right unit laws) If $x\in \mC$ then ${\ran}(x)x = x{\dom}(x) =x$. 
\end{enumerate}
We call the maps $\dom,\ran, {\pr}$ the {\em domain}, the {\em range}, and the {\em composition} (or {\em multiplication}) maps, respectively.
\end{definition}

A monoid can be viewed as a category where the set of units, $\mC^{(0)}$, is a singleton, $\{1\}$, which implies $\dom(x) = \ran(x) = 1$ for all $x\in \mC$ and $\mC^{(2)} = \mC\times \mC$. In the monoid case, axioms (1), (2) and (3) hold automatically; thus, only axioms (4) and (5) need to be imposed. A {\em groupoid} is a category where each arrow is invertible. That is, for each $x\in \mC$, there exists a unique $y\in \mC$ such that ${\ran}(y) = {\dom}(x)$ and ${\dom}(y) = {\ran}(x)$, and, in addition, $yx = {\dom}(x)$ and $xy = {\ran}(x)$. Such an element $y$ is called the {\em inverse} of $x$ and is denoted by $x^{-1}$. 

\begin{definition} (Functor) Let $\mC$ and $\mD$ be categories. A {\em functor} is a map $f\colon \mC\to \mD$  such that $f(\mC^{(0)}) \subseteq \mD^{(0)}$ and the following conditions hold:
\begin{enumerate}
\item (commutation of $f$ with $\dom$ and $\ran$) $f(\ran(x)) = \ran(f(x))$ and $f(\dom(x)) = \dom(f(x))$, for all $x\in \mC$.
\item (preservation of product) If $(x,y)\in \mC^{(2)}$ then $f(xy) = f(x)f(y)$.
\end{enumerate}
\end{definition}

We say that $f$ is an {\em isomorphism} if it is a bijective functor. If is readily checked that if $f\colon \mC\to \mD$ is an isomorphism, then $f|_{\mC^{(0)}}\colon \mC^{(0)}\to \mD^{(0)}$ is a bijection and, moreover, the inverse bijection $f^{-1}\colon \mD\to \mC$ is also an isomorphism.

\begin{definition} (Topological category) A {\em topological category} is a category endowed with a topology under which the maps $\dom, \ran \colon \mC\to \mC^{(0)}$ are continuous with respect to the relative topology on $\mC^{(0)}$ and the composition $\pr\colon \mC^{(2)}\to \mC$ is continuous with respect to the relative product topology on $\mC^{(2)}$.
\end{definition}

\begin{remark} The inclusion map $\unit\colon \mC^{(0)}\to \mC$ is automatically continuous with respect to the relative topology on $\mC^{(0)}$, since $\unit^{-1}(A) = A\cap \mC^{(0)}$ for any $A\subseteq \mC$.
\end{remark}

\begin{remark} In the definition of a topological groupoid (see, e.g., \cite{CH20,Exel08}) it is usually required that the multiplication  and the inversion maps are continuous. Since $\dom(x) = x^{-1}x$ and $\ran(x) = xx^{-1}$, the maps $\dom$ and $\ran$ are then also continuous. 
\end{remark}

By a {\em topological isomorphism} $f\colon \mC\to \mD$ between topological categories $\mC$ and $\mD$ we mean a functor $f$ which is also a homeomorphism.

\subsection{\'Etale and ample categories} The following definition was introduced in \cite[Definitions 6.1 and 7.1]{Kud25} and is inspired by the definition of an \'etale groupoid \cite{Paterson, R80,R07, St10}.

\begin{definition} (\'Etale, strongly \'etale and bi\'etale categories) A topological category $\mC$ is called:
\begin{enumerate} 
\item {\em \'etale} (resp. {\em co\'etale}) if the unit space $\mC^{(0)}$ is locally compact and Hausdorff in the relative topology and the domain map $\dom\colon \mC\to \mC^{(0)}$ (resp. the range map $\ran\colon \mC\to \mC^{(0)}$) is a local homeomorphism;
\item  {\em strongly \'etale}, if it is \'etale and the range map $\ran\colon \mC\to \mC^{(0)}$ is open;
\item {\em bi\'etale}, if it is both \'etale and co\'etale. 
\end{enumerate}
\end{definition}

If $\mC$ is a topological groupoid, we have the following (see \cite[Remark 6.2]{Kud25}):
$$
\mC \text{ is \'etale } \quad \Leftrightarrow  \quad \mC \text{ is co\'etale }  \quad \Leftrightarrow  \quad \mC \text{ is bi\'etale }\quad  
\Leftrightarrow \quad \mC \quad \text{ is strongly \'etale.}$$

Let $\mC$ be an \'etale category. It follows from \cite[Corollary 6.4]{Kud25} that $\mC^{(0)}$ is an open subset of $\mC$.\footnote{This generalizes the respective property of \'etale groupoids, see \cite[Proposition 3.2]{Exel08}.} Furthermore, Tristan Bice observed (unpublished) that the composition map $\pr\colon \mC^{(2)}\to \mC$ is open, see \cite[Proposition 6.9]{Kud25}.\footnote{This also generalizes the respective property of \'etale groupoids, \cite[Proposition 2.2.4]{Paterson}.}

The following definition \cite[Definition 6.6]{Kud25} is inspired by the notion of an open bisection in a topological groupoid. 
\begin{definition}(Slices, coslices and bislices) Let $\mC$ be a topological category. An open subset $U\subseteq \mC$ is called a {\em slice} (or an {\em open local section}) if the restriction $\dom|_U$ of the domain map $\dom$ to $U$ is a homeomorphism onto an open subset of $\mC^{(0)}$. It is called an  {\em coslice} (or an {\em open local cosection}) if $\ran|_U$ is a homeomorphism onto an open subset of $\mC^{(0)}$. It is called an  {\em bislice} (or an {\em open local bisection}) if it is both a slice and a coslice.
\end{definition}

Note that slices in a topological category (or even in a topological groupoid) do not in general coincide with bislices, as the following example shows.
\begin{example}
Let $X$ be a set, and consider the {\em pair groupoid} $X\times X$ (endowed with the discrete topology). Then its bislices are in a bijective correspondence with the elements of the symmetric inverse monoid $I(X)$. In contrast, its slices are in a bijective correspondence with the elements of the partial transformation monoid $PT(X)$. 
\end{example}

For subsets $U,V$ of a category $\mC$ we define their product by
\begin{equation}\label{eq:subs_mult}
UV = \{xy \colon x\in U, y\in V \text{ and } (x,y)\in \mC^{(2)}\}.
\end{equation}
We also define 
\begin{align*}
U^* & = \dom(U) = \{\dom(x) \colon x\in U\} \subseteq \mC^{(0)} \quad \text{ and }\\
U^+ & = \ran(U) = \{\ran(x) \colon x\in U\} \subseteq \mC^{(0)}.
\end{align*}

Let $\mC$ be an \'etale category. We denote the set of all its slices by $\mC^{op}$, and the set of all its bislices by $\widetilde{\mC}^{op}$. By \cite[Lemma 6.7]{Kud25}, $\mC^{op}$ forms a basis of a topology on $\mC$, Furthermore, if $\mC$ is bi\'etale, $\widetilde{\mC}^{op}$ forms a basis of a topology on $\mC$.  By \cite[Proposition 6.11]{Kud25}, the set $\mC^{op}$ is closed with respect to the composition and the unary operation $U\mapsto U^*$, forming a restriction monoid with the unit $\mC^{(0)}$. In particular, the projections of the restriction monoid $\mC^{op}$ are precisely the slices contained in $\mC^{(0)}$.

The following definition is taken from \cite[Definitions 7.1 and 7.7]{Kud25} and is motivated by the notion of an ample groupoid \cite[Definition 3.5, Proposition 3.6]{St10}.

\begin{definition} (Ample, strongly ample and biample categories) We say that an \'etale category $\mC$ is {\em ample} if compact slices form a basis of the topology. {\em Coample 
} categories are defined dually. A category $\mC$ is {\em strongly ample} if it is both strongly \'etale and ample. Further, $\mC$ is {\em biample} if compact bislices from a basis of the topology.
\end{definition}

It is shown in \cite[Lemma 7.2]{Kud25} that a bi\'etale category is biample if and only if it is both ample and coample. Moreover,   \cite[Proposition 7.3]{Kud25} shows that an \'etale category is ample if and only if $\mC^{(0)}$ is a locally compact Stone space, and a bi\'etale category is ample if and only if it is biample. Consequently, an \'etale groupoid is ample if and only if it is biample if and only if it is strongly ample. 

We denote the set of all compact slices of an ample category $\mC$ by $\mC^{a}$ and the set of all compact bislices by $\widetilde{\mC}^{a}$. 
Parts (1), (2) and (3) of the next proposition are proved in \cite[Propositions 7.6, 7.9, 7.10, 7.13]{Kud25}. The statement of part (4) is well known and is a consequence of Stone duality for Boolean inverse semigroups, see, e.g., \cite{L23}.   

\begin{proposition} \label{prop:ample} \mbox{}
\begin{enumerate}
\item Let $\mC$ be an ample category. Then $\mC^{a}$ is closed with respect to the composition and the unary operation $U\mapsto U^*$ and is a Boolean restriction semigroup with local units.
\item If $\mC$ is strongly ample, then $\mC^a$ is in addition closed with respect the unary operation $U\mapsto U^+$ and is a Boolean range semigroup.
\item If $\mC$ is biample, then the range semigroup $\mC^a$ is an \'etale Boolean range semigroup (see \cite{Kud25} for the definition), and $\widetilde{\mC}^a$ is a Boolean birestriction semigroup. 
\item If $\mC$ is an ample groupoid, $\widetilde{\mC}^a$ is a Boolean inverse semigroup.
\end{enumerate}
\end{proposition}

\section{\'Etale categories of partial monoid actions}\label{s:monoid_actions}
Let $X$ be a topological space. A {\em partial self-map} of $X$ is a map $f\colon {\mathrm{dom}}(f)\to X$ where ${\mathrm{dom}}(f)\subseteq X$ is the {\em domain} of $f$. By $\T(X)$ we denote the monoid of all continuous self-maps $X\to X$, and by $\PT(X)$ the monoid of all continuous partial self-maps of $X$ with open domains. Note that $\T(X)\subseteq \PT(X)$. If $X$ is endowed with the discrete topology,  $\T(X)$ coincides with the 
full transformation monoid $T(X)$, and $\PT(X)$ coincides with the partial transformation monoid  $PT(X)$.  
Note that $\PT(X)$ is a $(\cdot \,, ^*)$-subalgebra of $PT(X)$, where, for any $f\in \PT(X)$,  
$f^*$ is the identity map ${\mathrm{id}}_{\mathrm{dom}(f)}$
on ${\mathrm{dom}}(f)$. 

Let $S$ be a monoid and $T$ a restriction monoid. A map $\theta \colon S\to T$, $s\mapsto \theta_s$ is called a {\em premorphism} if the following conditions are satisfied:
\begin{enumerate}
\item $\theta_1 = 1$,
\item $\theta_s\theta_t \leq \theta_{st}$.
\end{enumerate}

\begin{definition} (Partial actions of a monoid on a locally compact Hausdorff space) Let $S$ be a monoid and $X$ a locally compact Hausdorff space. A {\em partial action} 
of $S$ on $X$ is a premorphism 
$\theta\colon S \to \PT(X)$. We also say that $S$ {\em acts partially} on $X$ via $\theta$.\footnote{For a comprehensive survey on partial actions, see \cite{D19}.}
\end{definition}

A partial action $\theta$ of $S$ on $X$ is equiva\-lently given by the pair $$\Phi = (\{{\mathrm{dom}}(\theta_s)\}_{s\in S}, \{\theta_s\}_{s\in S}),$$ where, for each $s\in S$, ${\mathrm{dom}}(\theta_s)$ is an open subset of $X$ and $\theta_s\colon {\mathrm{dom}}(\theta_s)\to X$ is a continuous map, satisfying, for all $s,t\in S$:
\begin{enumerate}
\item[(1)] $\theta_1 = {\mathrm{id}}_X$, 
\item[(2)] if $x\in {\mathrm{dom}}(\theta_t)$ and $\theta_t(x)\in {\mathrm{dom}}(\theta_s)$ then $x\in {\mathrm{dom}}(\theta_{st})$,
\item[(3)] if $x\in {\mathrm{dom}}(\theta_t)$ and $\theta_t(x)\in {\mathrm{dom}}(\theta_s)$ then $\theta_s\theta_t(x) = \theta_{st}(x)$.
\end{enumerate}

Actions arise as a special case of partial actions, as follows. 
\begin{definition} (Actions of a monoid on a locally compact Hausdorff space) Let $S$ be a monoid and $X$ a locally compact Hausdorff space. An {\em action} of $S$ on $X$ is a monoid morphism $\theta\colon S \to \T(X)$. \end{definition}

Suppose $S$ is a discrete monoid which acts partially on $X$ via $\theta \colon S \to \PT(X)$. Let $S\ltimes X$ (or $S\ltimes_{\theta} X$ to emphasise $\theta$) be the set of all $(s,x)\in S\times X$ such that $x \in {\mathrm{dom}}(\theta_s)$. We endow this set with the relative topology inherited from the product topology on $S\times X$.
For every $(s,x)\in S\ltimes X$, we set 
$$\dom(s,x) = (1,x) \quad  \text{ and } \quad \ran(s,x) = (1, \theta_s(x)).
$$
The product $(s,x)(t,y)$ is defined if and only if $\theta_t(y) = x$, in which case $(s,x)(t,y) = (st, \theta_{st}(y))$.
This turns $S\ltimes X$ into a topological category whose unit space is homeomorphic to $X$ via the map $(1,x)\mapsto x$.  The category $S\ltimes X$ is called the {\em partial transformation category} of the partial action $\theta$. 
If $\theta$ is an action, the category $S\ltimes X$ is denoted by $S\times X$ (since its underlying set coincides with $S\times X$) and is called the {\em transformation category} of the action $\theta$.

\begin{proposition} \label{prop:p1} Suppose $S$ is a monoid which acts partially on $X$ via $\theta$. 
\begin{enumerate}
\item If $Y$ is an open subset of ${\mathrm{dom}}(\theta_s)$ then $(s,Y) = \{(s,y)\colon y\in Y\}$ is a slice homeomorphic to $Y$ via the map $(s,y)\mapsto y$.
\item $S\ltimes X$ is an \'etale category.
\item The sets $(s, Y)$, where $s\in S$ and $Y$ is an open subset of  ${\mathrm{dom}}(\theta_s)$, form a basis of the topology on $S\ltimes X$.
\item If $\theta$ acts by open maps, the category $S\ltimes X$ is strongly \'etale.
\item If $\theta$ acts by homeomorphisms, the category $S\ltimes X$ is bi\'etale.
\end{enumerate}
\end{proposition}

\begin{proof}
(1) The image of $(s,Y)$ under $\dom$ is $(1,Y)$, which is obviously homeomorphic to $Y$. 

(2) Let $(s,x) \in S\ltimes X$ where $x\in {\mathrm{dom}}(\theta_s)$. Then $(s,{\mathrm{dom}}(\theta_s))$ is a neighbourhood of $(s,x)$, which is homeomorphic to $(1, {\mathrm{dom}}(\theta_s))$ via  $\dom$.

Part (3) is immediate by the definition of the product topology.

For part (4), let $(s,U)$ be a basic open set. Then $\ran(s,U) = (1, \theta_s(U))$ which is an open set since $\theta_s$ is open.

Finally, for part (5), let $(s,x) \in S\ltimes X$ where $x\in {\mathrm{dom}}(\theta_s)$.  Then $(s,{\mathrm{dom}}(\theta_s))$ is a neighbourhood of $(s,x)$ and homeomorphic to  $(1, {\mathrm{ran}}(\theta_s))$ via $\ran$. 
\end{proof}  

\begin{example}~\hspace{-1em}\footnote{The author gratefully acknowledges Ruy Exel for suggesting this example.} Let $X$ be the one-point compactification of the discrete space ${\mathbb N}_0$. Specifically, $X = {\mathbb N}_0 \cup \{\infty\}$, where the basic open sets are the finite subsets of ${\mathbb N}_0$ and the cofinite subsets of $X$ containing $\infty$. It is a compact Hausdorff space.
We consider the action $\theta\colon {\mathbb N}_0 \to \T(X)$ of the monoid ${\mathbb N}_0=\{1,s,s^2, \dots\}$ on $X$, generated by the continuous map $\theta_s\colon X \to X$ defined by
$$
\theta_s(x) = \left\lbrace \begin{array}{ll} \left\lfloor \frac{x}{2} \right\rfloor, & x\in {\mathbb N}_0;\\
\infty, & x= \infty.
 \end{array}\right.
$$
By Proposition \ref{prop:p1}, this action gives rise to the \'etale transformation category ${\mathbb N}_0\times X$. The category ${\mathbb N}_0\times X$ is, however, not co\'etale. This follows from the observation that any neighbourhood of $(s,\infty)$ contains some open set $(s,U)$, where $U$ is a cofinite subset of ${\mathbb N}_0$, and therefore the restriction of $\theta_s$ to $U$ is not injective. Consequently, the restriction of $\ran$ to $(s,U)$ is not injective, thus $\ran$ is not a local homeomorphism. Despite the category is not co\'etale, it is strongly \'etale, since $\theta_s$ (and thus all  $\theta_{s^k}$) maps basic open sets to basic open sets and is thus an open map. 
\end{example} 

\section{Actions of restriction semigroups and their categories of germs}\label{s:germs}
\subsection{Actions of restriction semigroups} 
The following definition is inspired by  \cite[Definition 4.3]{Exel08} and \cite[Definition 5.1]{St10}.

\begin{definition} \label{def:act_restr} (Actions and ample actions of restriction semigroups) \label{def:action_restr} Let $S$ be a restriction semigroup and $X$ a locally compact Hausdorff space. An {\em action} of $S$ on $X$ is a $(\cdot\,,^*)$-morphism $\theta\colon S\to \PT(X)$. 
We say that an action $\theta$ is {\em ample} if $X$ is a locally compact Stone space and ${\mathrm{dom}}(\theta_s)\in {\mathcal B}(X)$ for all $s\in S$.
\end{definition}

Since, for every $e\in P(S)$, the map $\theta_e$ is a projection, it is the identity map ${\mathrm{id}}_{X_e}$ on its domain $X_e = {\mathrm{dom}}(\theta_e)$.
Throughout the paper, we will always assume that $\theta$ is {\em non-degenerate} which means that 
\begin{equation}\label{eq:non_deg}
\bigcup_{e\in P(S)} X_e = X.
\end{equation}

Since ${\mathrm{dom}}(f) = {\mathrm{dom}}(f^*)$ for every $f\in \PT(X)$, and since $\theta$ is a $(\cdot\,, ^*)$-morphism, for every $s\in S$ we have:
\begin{equation}\label{eq:e1}
{\mathrm{dom}}(\theta_s) = {\mathrm{dom}}(\theta_s^{\,*}) = {\mathrm{dom}}(\theta_{s^*}) = X_{s^*}.
\end{equation}
Furthermore, if $e,f\in P(S)$, we have
$\theta_{ef} = \theta_e\theta_f = {\mathrm{id}}_{X_e} {\mathrm{id}}_{X_f} =  {\mathrm{id}}_{X_e\cap X_f}$.

We now provide two important examples of actions of restriction semigroups.

\begin{example} \label{ex:spectral} (The spectral action) Let $S$ be a restriction semigroup with local units (see Definition \ref{def:lu}) and $X=\widehat{P(S)}$ be the character space of $P(S)$. For each $s\in S$ we set
$\beta_s$ to be the partial self-map of $X$ with 
$${\mathrm{dom}}(\beta_s) = \{\varphi \in X \colon \varphi(s^*) = 1\},$$ 
and 
\begin{equation*}
\beta_s(\varphi)(e) = \varphi((es)^*), \quad \varphi \in {\mathrm{dom}}(\beta_s).
\end{equation*}
The set ${\mathrm{dom}}(\beta_s)$ is precisely the set $D_{s^*}$, defined in \eqref{eq:de}, and we will use this notation for ${\mathrm{dom}}(\beta_s)$ rather than $X_{s^*}$ or $\widehat{P(S)}_{s^*}$, to emphasise that we are working with the spectral action.

Since $S$ has local units, there is $e\in P(S)$ such that $es=s$, which means $\varphi((es)^*) = 1$, and thus the map $\beta_s(\varphi)$ is non-zero. Furthermore, for $e,f\in P(S)$ we have $(es)^*(fs)^* = (es(fs)^*)^* = (efs)^*$, which yields that $\beta_s(\varphi)\in X$.
The set ${\mathrm{dom}}(\beta_s)$ is open by the definition of the topology on $X$. We show that $\beta_s$ is continuous. It suffices to show that $\beta_s^{-1}(D_e)$ and $\beta_s^{-1}(D_{e;f_1,\dots f_k})$ are open for all basic open sets $D_{e;f_1,\dots, f_k}$, where $k\geq 0$. 
But 
$$\varphi\in \beta_s^{-1}(D_e) \quad \Leftrightarrow \quad \beta_s(\varphi)(e) = 1 \quad \Leftrightarrow \quad \varphi ((es)^*) = 1 \quad \Leftrightarrow \quad \varphi \in D_{(es)^*}.
$$ Hence, $\beta_s^{-1}(D_e) = D_{(es)^*}$ and $\beta_s^{-1}(D_{e;f_1,\dots, f_k}) = D_{(es)^*;(f_1s)^*,\dots, (f_ks)^*}$, which proves that $\beta_s$ is continuous.

To show that the map $s\mapsto \beta_s$ is a morphism of restriction semigroups, observe that ${\mathrm{dom}}(\beta_{st}) = \{\varphi \in X \colon \varphi((st)^*)= 1\}$ and
$\varphi \in {\mathrm{dom(\beta_s\beta_t)}}$ if and only if $\varphi(t^*) = 1$ and $\beta_t(\varphi) \in {\mathrm{dom(\beta_s)}}$, which means that  $1 = \beta_t(\varphi)(s^*)= \varphi((s^*t)^*) = \varphi((st)^*)$. Since $(st)^* = (s^*t)^*\leq t^*$ and $\varphi((st)^*)=1$, it follows that $\varphi(t^*)=1$. This implies that ${\mathrm{dom}}(\beta_{st}) = {\mathrm{dom(\beta_s\beta_t)}}$. If $\varphi \in {\mathrm{dom}}(\beta_{st})$ and $e\in P(S)$, we have $\beta_{st}(\varphi)(e) = \beta_s\beta_t(\varphi)(e) = \varphi((est)^*)$, so that $\beta_{st} = \beta_s\beta_t$. Furthermore, $\beta_{t^*}$ is the identity map on $D_{t^*} = \{\varphi\in X\colon \varphi(t^*)=1\}$, which is the domain of $\beta_t$. Hence $\beta_{t^*} = \beta_t^*$, and we have shown that $\beta$ is a $(\cdot\,,^*)$-morphism. \end{example}

\begin{example} \label{ex:action_ample}
Let ${\mathcal C}$ be an ample category. We define the action $\theta$ of ${\mathcal C}^{a}$ on ${\mathcal C}^{(0)}$ as follows.
If $U\in {\mathcal C}^a$ and $x\in {\mathcal C}^{(0)}$, then $\theta_U(x) = U\circ x$ is defined if and only if $x\in \dom(U)$. If this is the case, we set  $U\circ x= \ran(Ux)$, where $Ux$ is the product in ${\mC}$, which equals the only $u\in U$ with $\dom(u) = x$. Therefore, we have:
$$
\theta_U\colon \dom(U) \to \ran(U), \quad x\mapsto \ran(Ux).
$$
We show that $\theta_U\theta_V=\theta_{UV}$. Suppose that $x\in {\mathrm{dom}}(\theta_{UV}) = \dom(UV)$. Then  $x = \dom(UVx)$. Then $\ran(Vx) \in \dom(U)$ and
$\ran(UVx) = \ran(U\ran(Vx))$, which yields $x\in {\mathrm{dom}}(\theta_U\theta_V)$ and $UV\circ x = U\circ (V\circ x)$. Clearly, we have ${\mathrm{dom}}(\theta_U\theta_V)\subseteq {\mathrm{dom}}(\theta_{UV})$.
Furthermore, as $U^* = \dom(U)$, we have that both
$\theta(U)^*$ and $\theta(U^*)$ coincide with the identity map on $\dom(U)$. It follows that $\theta$ is an action.
\end{example}

\subsection{The category of germs} \label{subs:germs}
Following Exel \cite{Exel08}, we define\footnote{We would like to warn the reader that $s^*$ denotes the inverse of $s$ in an inverse semigroup $S$ in \cite{Exel08}. The element $s^*s$ of \cite{Exel08} is thus an analogue for our element $s^*$.}
$$
\Omega = \{(s,x)\in S\times X \colon x\in X_{s^*}\}.
$$
For every $(s,x), (t,y)\in \Omega$ we write $(s,x)\sim (t,y)$ if $x=y$ and there exists $u\in S$ such that $x\in X_{u^*}$ and $u\leq s,t$. 

\begin{lemma} $\sim$ is an equivalence relation.
\end{lemma}
\begin{proof} Only transitivity needs a proof. Suppose that $(s,x)\sim (t,x) \sim (v,x)$. Then there are $p\leq s,t$ and $q\leq t,v$ with $x\in X_{p^*}\cap X_{q^*}$. Since $p$ and $q$ have a common upper bound, $t$, we have $p\smile q$ by \cite[Lemma 4.15]{Kud25}.  Let $u=pq^* = qp^*$. Then $u\leq p,q$, so that $u \leq s,t,v$. In addition, $u^*= (pq^*)^* = p^*q^*$, using \eqref{eq:axioms_star}. It follows that $x\in X_{p^*}\cap X_{q^*} = X_{p^*q^*} = X_{u^*}$. Hence $(s,x) \sim (v,x)$.
\end{proof}

The equivalence class of $(s,x)$ will be called the {\em germ of} $s$ {\em at} $x$ and will be denoted $[s,x]$.
\begin{lemma} \label{lem:l2} If $(s,x)\sim (t,x)$ then $\theta_s(x) = \theta_t(x)$.
\end{lemma}

\begin{proof}
By assumption, there is $u\leq s,t$ such that $x\in X_{u^*}$. It suffices to show that $\theta_s(x)= \theta_u(x)$. Since $u\leq s$, we have $u=su^*$. Using $\theta_{u^*}(x) = x$, we obtain $\theta_u(x) = \theta_{su^*}(x) = \theta_s\theta_{u^*}(x) = \theta_s(x)$, as desired.
\end{proof}

The following is an analogue of \cite[Proposition 4.7]{Exel08}.

\begin{proposition} \label{prop:p3}
If $(s,x), (t,y)\in \Omega$ are such that $x=\theta_t(y)$, then $(st, y)\in \Omega$. Furthermore, if $(p,x) \sim (s,x)$ and $(q,y)\sim (t,y)$ satisfy $x=\theta_t(y)$, we have that $(st,y)\sim (pq,y)$. \end{proposition}

\begin{proof}
By assumption, $x\in X_{s^*}$, so that $x\in {\mathrm{dom}}(\theta_s)$. It follows that $y\in {\mathrm{dom}}(\theta_s\theta_t) = {\mathrm{dom}}(\theta_{st}) = X_{(st)^*}$, so that $(st,y)\in \Omega$. For the second claim, let $u\leq s,p$ and $v\leq t,q$ be such that $x\in X_{u^*}$ and $y\in X_{v^*}$. Then $x\in {\mathrm{dom}}(\theta_u)$ and $y\in {\mathrm{dom}}(\theta_v)$. Since $(v,y)\sim (t,y)$, Lemma \ref{lem:l2} implies that $\theta_v(y) = \theta_t(y)=x$. It follows that $y \in {\mathrm{dom}}(\theta_u\theta_v) = {\mathrm{dom}}(\theta_{uv}) = X_{(uv)^*}$. In addition, since $u\leq s$ and $v\leq t$, we have $uv\leq sv \leq st$ by Lemma \ref{lem:1}, so that $(uv,y)\sim (st,y)$. By symmetry, we also have  $(uv,y)\sim (pq,y)$, and the statement follows.
\end{proof}

Let  $S\ltimes_{\theta} X  = \Omega/\sim$ be the set of germs.\footnote{Since it is always clear from the context if $S$ is assumed to be a monoid or a restriction semigroup, there is no danger of confusion between the category of germs $S\ltimes X$ of a an action of a restriction semigroup $S$ and the partial transformation category $S\ltimes X$ of a partial action of a monoid $S$.} We abbreviate  $S\ltimes_{\theta} X$  to  $S\ltimes X$ if $\theta$ is understood. We set
$$\mC = S\ltimes_{\theta} X
$$
and define the maps $\dom,\ran \colon \mC \to \mC$  by
\begin{equation}\label{eq:oper1}
\dom([s,x]) = [s^*,x], \qquad \ran([s,x]) = [e, \theta_s(x)],
\end{equation}
where $e\in P(S)$ is such that $\theta_s(x)\in X_e$. We aim to show that $\mC$ is a category. The set $\mC^{(0)}$ then consist of all germs of the form $[x,e]$ with $e\in P(S)$.
Furthermore, the set of composable pairs is the set
$$
\mC^{(2)} = \{([s,x], [t,y]) \in \mC\times \mC
\colon x=\theta_t(y)\}.$$
For $([s,x], [t,y]) \in \mC^{(2)}$ we set
\begin{equation}\label{eq:def_mult}
\pr([s,x], [t, y]) = [st,y],
\end{equation}
 which is well defined by Proposition \ref{prop:p3}.
\begin{lemma}
$\mC$ is a category. 
\end{lemma}

\begin{proof} The proof amounts to routine verification of the axioms of Definition \ref{def:cat}. We check that $\ran(a)a=a$ holds for all $a\in \mC$. Suppose that $a=[s,x]$. Then $\ran(a) = [e, \theta_s(x)]$, where $\theta_s(x)\in X_e$. Then $\ran(a)a = [e,\theta_s(x)][s,x] = [es,x]$. Since $es\leq s$ and $x\in X_{(es)^*}$, we have $[es,x] = [s,x]$, so that $\ran(a)a=a$. The remaining verifications can be carried out similarly.
\end{proof}

We now topologize the category $\mC$. For $s\in S$ and an open subset $U\subseteq X_{s^*}$ we put
$$
(s,U) = \{[s,x]\in \mC \colon x\in U\}.
$$

The following is an analogue of \cite[Proposition 4.14]{Exel08}.

\begin{proposition}
Let $s,t\in S$, and let $U,V$ be open sets such that $U\subseteq X_{s^*}$ and $V\subseteq X_{t^*}$. If $[u,x] \in (s,U)\cap (t,V)$, there are $v\leq u$ and an open set $W\subseteq X_{v^*}$ such that $[u,x]\in (v, W) \subseteq (s, U) \cap (t,V)$. 
\end{proposition}

\begin{proof}
The assumption $[u,x] \in (s,U)\cap (t,V)$ implies that $[u,x] = [s,y] = [t,z]$ for some $x,y,z\in X$. The definition of equality of germs implies that $x=y=z$, meaning that $x\in U\cap V$ and $[u,x] = [s,x] = [t,x]$. Furthermore, there are $v_1\leq u,s$ and $v_2\leq u,t$ such that $x\in X_{v_i^*}$ for $i=1,2$ and $[u,x] = [v_1,x] = [v_2,x]$. It follows that there is $v\leq v_1,v_2 \leq s,t$ such that $x\in X_{v^*}$ and $[u,x] = [v,x]$. Thus,
$[s,x] = [t,x] = [v,x]$. We set $W= U \cap V \cap X_{v^*}$. Since $x\in W$, we have $[v,x] \in (v,W)$. To prove that $(v,W) \subseteq (s,U) \cap (t,V)$, let $[v,y]$ be an arbitrary element of $(v,W)$. Then $y\in W\subseteq U$. Since $v\leq s$, we have $[v,y] = [s,y]$, which implies that $[v,y]\in (s,U)$. A symmetric argument gives $[v,y]\in (t,V)$. Therefore, $[v,y]\in (s,U)\cap (t,V)$ and, consequently, $(v,W)\subseteq (s,U)\cap (t,V)$.
\end{proof}

It follows that the collection of all $(s,U)$, where $s\in S$ and $U\subseteq X_{s^*}$ is an open set, forms a basis of a topology on $\mC$. From now on, we will consider $\mC$ as a topological space equipped with this topology. 

\begin{lemma}\label{lem:homeo}
The germs of the form $[e,x]$, where $e\in P(S)$, are determined by their second component $x$ and do not depend on $e$. Moreover, the map 
$$\mC^{(0)}\to X, \quad [e,x]\mapsto x$$
is a homeomorphism.
\end{lemma}

\begin{proof}
If $x\in X$ and $e,f\in P(S)$ are such that $x\in X_e\cap X_f$, then $x\in X_{ef}$, so that $[e,x] = [f,x]$ since $ef \leq e,f$.
This proves the first claim and shows that the map $\psi\colon [e,x]\mapsto x$ is a bijection. 
Since $\psi(e,U) = U$, where $U\subseteq X_e$ is a basic open set, $\psi$ is an open map.
Let $U$ be an open subset of $X$. Since the action is non-degenerate, for each $x\in U$ there is $e_x\in P(S)$ such that $x \in X_{e_x}$.
Setting $U_x = U \cap X_{e_x}$, we have that $U = \cup_{x\in U} U_x$. Since $\psi^{-1}(U_x) = (e_x, U_x)$, it follows that the map $\psi$ is continuous. Hence, it is a homeomorphism.
\end{proof}

The following analogue of \cite[Proposition 4.18]{Exel08}.

\begin{proposition} \label{prop:a21} Let $s\in S$ and $U\subseteq X_{s^*}$ is an open set. Then $(s,U)$ is a slice homeomorphic to $U$.
\end{proposition}

\begin{proof} Bearing in mind the definition of $\dom$, it follows that $(s,U)$ is a slice. We have that $\dom(s, X_{s^*}) = (s^*, X_{s^*})$, which, in view of Lemma \ref{lem:homeo}, is homeomorphic to $X_{s^*}$. If suffices, therefore, to prove that the map 
$$\dom|_{(s, U)} \colon (s, U) \to U,\quad [s,x]\mapsto x
$$
is a homeomorphism with respect to relative topologies on $(s, U)$ and $U$. This map is clearly a bijection. It is continuous because $\dom$ is continuous and $X_{s^*}$ is open in $X$.
Since its inverse map is also continuous, it is a homeomorphism.
\end{proof}

We now prove an analogue of \cite[Proposition 4.14]{Exel08}.

\begin{proposition} \label{prop:etale}$\mC$ is an \'etale topological category.
\end{proposition}

\begin{proof} 
We first show that the  maps $\dom$ and  $\ran$ are continuous. Let $e\in P(S)$ and $U\subseteq X_e$ be an open set. To show that $\dom^{-1}(e,U)$ is open, we observe that
$$
[s,x]\in \dom^{-1}(e,U) \quad \Leftrightarrow \quad [s^*,x] \in (e,U) \quad \Leftrightarrow \quad x\in U\cap X_{s^*}.
$$
Therefore, $(s, U\cap X_{s^*})$ is a neighbourhood of $[s,x]$ contained in  $\dom^{-1}(e,U)$.
Consequently, $\dom$ is continuous. To show that $\ran$ is continuous, we observe that
$$
[s,x]\in \ran^{-1}(e,U) \quad \Leftrightarrow \quad [f, \theta_s(x)] \in (e,U) \quad \Leftrightarrow  \quad x\in \theta_s^{-1}(U)\cap X_{s^*},
$$
where $f$ above is any projection such that $\theta_s(s)\in X_f$. Thus $(s, \theta_s^{-1}(U)\cap X_{s^*})$ is a neighbourhood of $[s,x]$ contained in $\ran^{-1}(e,U)$.
It follows that $\ran$ is continuous, too.

We finally show that $\pr$ is continuous. 
Suppose that $$\pr([u,\theta_v(x)], [v,x]) \in (s, U) \quad  \text{ where } \quad U\subseteq X_{s^*} \,\, \text{ is open}.$$
Then $[uv,x]\in (s,U)$, so that $[uv,x] = [s,x]$. The definition of the equality of germs implies that there is $e\in P(S)$ such that $uve = se$, meaning that $[uv,x] = [uve, x]=[se,x] = [s,x]$.
Then $x\in X_{(uve)^*}\cap U$. Since $(uve)^* = (u^*ve)^* \leq (ve)^*$, it follows that $X_{(uve)^*}\subseteq X_{(ve)^*}$.
We have that $A=(u, X_{u^*})$ is a neighbourhood of $[u, \theta_v(x)]$, $B=(ve, X_{(uve)^*}\cap U)$ is a neighbourhood of $[v,x] = [ve,x]$, and
$$\pr((A \times B) \cap \mC^{(2)}) \subseteq (uve, X_{(uve)^*}\cap U) = (se, X_{(se)^*}\cap U).$$
If $[se,x] \in (se, X_{(se)^*}\cap U)$ then $x\in X_{(se)^*} \subseteq X_{s^*}$ so that $[se,x] = [s,x] \in (s,U)$. It follows that 
$$\pr((A \times B) \cap \mC^{(2)}) \subseteq (s,U),$$
which implies that the map $\pr$ is continuous.

Finally, if $[s,x]\in \mC$, the slice $(s, X_{s^*})$ is its neighbourhood. Proposition \ref{prop:a21} implies that $\dom|_{(s, X_{s^*})}\colon (s,X^*) \to \dom(s, X_{s^*})$ is a homeomorphism. 
\end{proof}

We have the following immediate corollary.

\begin{corollary} Suppose that $\theta\colon S\to \PT(X)$ is an ample action. Then $\mC = S\ltimes_{\theta} X$ is an ample category. 
\end{corollary}

\begin{definition} (The category of germs of an action of a restriction semigroup)
The constructed \'etale category  $\mC=S\ltimes_{\theta} X$  will be called the {\em category of germs of the action} $\theta$ {\em of} $S$ {\em on} $X$.
\end{definition}

We now provide a morphism from $S$ to the restriction semigroup $\mC^{op}$.

\begin{proposition} \label{prop:morphism} Let $\theta\colon S\to \PT(X)$ be an action of $S$ on $X$. Then the map 
$$\Theta\colon S\to \mC^{op}, \quad s\mapsto (s, X_{s^*})
$$ is a morphism of restriction semigroups.
If the action is ample then the slice $(s, X_{s^*})$ is compact, so that the image of $\Theta$ is contained in the restriction semigroup $\mC^a$.
\end{proposition}

\begin{proof}
We first show that 
\begin{equation}\label{eq:prod25a}
(s, X_{s^*})(t,X_{t^*}) = (st, X_{(st)^*}).
\end{equation}
Suppose that $[s,y]\in (s, X_{s^*})$ and $[t, x]\in (t, X_{t^*})$ are composable. Then $y=\theta_t(x)$ where $x\in X_{t^*} = {\mathrm{dom}}(\theta_t)$ and
$y\in {\mathrm{dom}}(\theta_s)$. It follows that $x\in {\mathrm{dom}}(\theta_s\theta_t) = {\mathrm{dom}}(\theta_{st})$. Hence $[s,y][t,x] = [st,x] \in (st, X_{(st)^*})$.
Conversely, suppose that $[st, u]\in (st, X_{(st)^*}) = \mathrm{dom}(\theta_s\theta_t)$. Then $u\in {\mathrm{dom}}(\theta_t)$ and $\theta_t(u) \in {\mathrm{dom}}(\theta_s)$,
which shows that $$[st,u] = [s,\theta_t(u)][t,u]\in (s, X_{s^*})(t,X_{t^*}).$$
Furthermore, 
\begin{align*} \Theta(s)^*=(s,X_{s^*})^* & = \{\dom([s,x])\colon [s,x] \in (s,X_{s^*})\}\\
&  = \{[s^*,x]\colon [s,x] \in (s,X_{s^*})\} \\
&= (s^*, X_{s^*}) = \Theta(s^*), 
\end{align*} 
so that the map $\Theta$ preserves the $^*$  operation. 

Suppose that the action $\theta$ is ample. Since $(s, X_{s^*})$ is homeomorphic to $X_{s^*}$ which is compact, $(s, X_{s^*})\in \mC^a$, as desired.
\end{proof}

\subsection{Actions of range and birestriction semigroups}
Recall that $\PT(X)$ is the restriction semigroup of all continuous partial self-maps $\varphi$ of $X$ such that ${\mathrm{dom}}(\varphi)$ is open,
and the action of a restriction semigroup $S$ is defined as a $(\cdot\,,^*)$-morphism $S\to \PT(X)$ (see Definition \ref{def:act_restr}).
We set $$\PT_{o}(X) = \{\psi\in \PT(X) \colon \psi \text{ is open}\}.$$
Then $\PT_{o}(X)$ is a $(\cdot\,,^*)$-subalgebra of $\PT(X)$ and, moreover, it is closed with respect to the operation $\psi \mapsto {\mathrm{id}}_{{\mathrm{ran}}(\psi)}$. 
Hence it is a $(\cdot\,,^*,^+)$-subalgebra of the range semigroup $PT(X)$ of all partial self-maps of $X$, with
the unary operations $^*$ and $^+$ given by
$$\psi^* = {\mathrm{id}}_{{\mathrm{dom}}(\psi)} \quad \text{ and } \quad \psi^+ = {\mathrm{id}}_{{\mathrm{ran}}(\psi)}.$$
Since $PT(X)$ is a range monoid, so is $\PT_{o}(X)$.
Furthermore, 
$$\I(X) = \{\psi \in \colon \PT_{o}(X)\colon \psi \text{ is a homeomorphism from } {\mathrm{dom}}(\psi) \text{ to } {\mathrm{ran}}(\psi)\}
$$
is an inverse monoid, where $\psi^{-1}$ is the inverse homeomorphism to $\psi$. We say that elements of $\I(X)$ are {\em partial homeomorphisms} of $X$.
Being an inverse monoid, $\I(X)$ is also a birestriction monoid and is a $(\cdot\,,^*,^+,1)$-subalgebra of $\PT_{o}(X)$. 

We are now prepared to define actions of range and birestriction semigroups.

\begin{definition} \label{def:act_restr1} (Actions of range and birestriction semigroups) \label{def:action_range} Let $X$ be a locally compact Hausdorff space.
\begin{enumerate}
\item An {\em action} of a range semigroup $S$ on $X$ is a $(\cdot\,,^*,^+)$-morphism 
$$\theta\colon S\to \PT_{o}(X).$$
 \item An {\em action} of a birestriction semigroup $S$ by {\em partial homeomorphisms} is a $(\cdot\,,^*,^+)$-morphism $$\theta\colon S\to \I(X).$$
 \end{enumerate}
\end{definition}
We will consider only non-degenerate actions, that is, we assume that condition \eqref{eq:non_deg} holds.
Non-degenerate actions of birestriction semigroups have first appeared in \cite[Definition~2.21]{MC24} under the name {\em \'etale actions}.\footnote{Beware that in \cite{MC24}, birestriction semigroups are referred to as restriction semigroups.}

\begin{remark} Since a range or a birestriction semigroup $S$ is a restriction semigroup, if one disregards the $^+$ operation, one can also consider its more general actions in the sense of Definition \ref{def:action_restr}. In that case, however, $\theta(S)$ may not be contained in $\PT_{o}(X)$ or in $\I(X)$. Unless explicitly stated otherwise, an action of a range semigroup will, henceforth, refer to an action in the sense of Definition \ref{def:action_range}.
\end{remark}

For an action $\theta\colon S\to \PT_{o}(X)$ of a range semigroup, the definition implies that
$${\mathrm{dom}}(\theta_s) = {\mathrm{dom}}(\theta_{s^*}) = X_{s^*} \quad \text{ and } \quad {\mathrm{ran}}(\theta_s) = {\mathrm{ran}}(\theta_{s^+}) = X_{s^+}.$$

We recall the following well known definition, see, e.g., \cite[Definition 4.1]{St10}.

\begin{definition} \label{def:act_inv} (Actions of inverse semigroups) An {\em action of an inverse semigroup} $S$ on a locally compact Hausdorff space  $X$ is  a semigroup morphism $\theta \colon S\to \I(X)$. 
\end{definition}

The following proposition shows that this can be equivalently defined as an action of $S$, viewed as a restriction semigroup, that is, as a $(\cdot\, ,^*)$-morphism  $\theta\colon S\to \PT(X)$.

\begin{proposition} \label{prop:inv_s15} Suppose that $S$ is an inverse semigroup, and let $\theta\colon S\to \PT(X)$ be its action as a restriction semigroup given in Definition \ref{def:act_restr}, where $s^* = s^{-1}s$. Then $\theta(S)$ is contained in $\I(X)$. Consequently, $\theta\colon S\to \I(X)$ is an action of $S$ in the sense of Definition~\ref{def:act_inv}.
\end{proposition}

\begin{proof} 
Let $s\in S$. The equalities $\theta_s\theta_{s^{-1}}\theta_s$ and $\theta_{s^{-1}}\theta_s\theta_{s^{-1}} = \theta_{s^{-1}}$ imply that $\theta_s$ and $\theta_{s^{-1}}$ are mutually inverse bijections. Moreover, if $U\subseteq X_{s^*}$ is open then $\theta_s(U) = \theta_{s^{-1}}^{-1}(U)$, which is open as $\theta_{s^{-1}}$ is continuous. Therefore, $\theta_s \in \I(X)$.
\end{proof}

\begin{proposition} \label{prop:s16a} If $S$ is a range semigroup, then $\beta(S) \subseteq \PT_o(\widehat{P(S)})$; and if $S$ is a birestriction semigroup, then $\beta(S) \subseteq \I(\widehat{P(S)})$.
\end{proposition}

\begin{proof} Suppose first that $S$ is a range semigroup and let $s\in S$. In view of Example \ref{ex:spectral}, we only need to show that $\beta_s$ is an open map.
We have that ${\mathrm{dom}}(\beta_s) = D_{s^*}$. It suffices to show that $\beta_s(D_e)$ is open, where $e\leq s^*$. Observe that $D_{e} = D_{(es)^*}$. For every
$\varphi \in D_{e}$ we have $[s,\varphi] = [es, \varphi]$, whence $\beta_s(\varphi) = \beta_{es}(\varphi)$. Setting $t=es$, we need to show that $\beta_t(D_{t^*})$ is an open set.
We show that $\beta_t(D_{t^*}) =  D_{t^+}$. We have 
\begin{align*} \beta_t(\varphi) \in D_{t^+} & \Leftrightarrow \beta_t(\varphi)(t^+) = 1 \\
& \Leftrightarrow \varphi((t^+t)^*) = 1 \\
& \Leftrightarrow \varphi(t^*) = 1 \Leftrightarrow \varphi \in D_{t^*},
\end{align*}
as desired. 

Suppose now that $S$ is birestriction. We aim to show that $\beta_s\colon D_{s^*}\to D_{s^+}$, where $s\in S$, is a homeomorphism. We already know that it is an open and continuous map. We show that it is bijective. Define the map 
$$
\gamma_s\colon D_{s^+} \to D_{s^*}, \quad \gamma_s(\varphi)(e) = \varphi((se)^+).
$$
Let $\varphi \in D_{s^*}$ and $e\in P(S)$. Then:
\begin{align*}
\gamma_s\beta_s(\varphi)(e) & = \beta_s(\varphi)((se)^+) & \text{by the definition of } \gamma_s \\
& = \varphi(((se)^+s)^*) & \text{by the definition of } \beta_s\\
& = \varphi((se)^*)  & \text{by } \eqref{eq:ample_r}\\ 
& = \varphi(s^*e) & \text{by } \eqref{eq:rule1r} \\
& = \varphi(s^*)\varphi(e) = \varphi(e) & \text{since } \varphi \text{ is a morphism and } \varphi(s^*) =1.
\end{align*}
This shows that $\gamma_s\beta_s(\varphi) = \varphi$. By symmetry, we also have $\beta_s\gamma_s(\varphi) = \varphi$ for all $\varphi\in D_{s^+}$.
Therefore, $\beta_s$ is a bijection and $\gamma_s = \beta_s^{-1}$. It follows that $\beta_s$ is a homeomorphism.
\end{proof}

The definition of a spectral action of a birestriction semigroup was first introduced in \cite{MC24}.

We can now refine Proposition \ref{prop:etale} by providing extra properties of categories of germs of actions of range and birestriction semigroups.

\begin{proposition} \label{prop:range27} Let $X$ be a locally compact Hausdorff space.
\begin{enumerate}
\item Let $S$ be a range semigroup and $\theta\colon S\to \PT_{o}(X)$ an action of $S$ on $X$. Then the category of germs $\mC=S\ltimes X$ is strongly \'etale. Moreover,
the restriction monoid $\mC^{op}$ is a range monoid where $U^+ = \ran(U)$ for every $U\in \mC^{op}$.
If $\theta$ is ample, then the category $\mC$ is strongly ample and $\mC^{a}$ is a Boolean range semigroup. 
\item Let $S$ be a birestriction semigroup and $\theta\colon S\to \I(X)$ an action of $S$ on $X$ by partial homeomorphisms. Then the category of germs $\mC=S\ltimes X$ is  bi\'etale. Moreover, the set of bislices $\widetilde{\mC}^{op}$ forms a birestriction monoid.  If $\theta$ is ample, then the category $\mC$ is biample and $\widetilde{\mC}^{a}$ is a Boolean birestriction semigroup.
\end{enumerate}
\end{proposition}

\begin{proof} (1) Let $X$ be an open subset of $X_{s^*}$. Since $\theta_s$ is an open map, the set $\ran(s,X) = \theta_s(X)$ is open.
Thus $\mC$ is strongly \'etale. The remaining claims follow applying \cite[Proposition 7.9]{Kud25}.

(2) Let $[s, x] \in \mC$, then $(s, X_{s^*})$ is a neighbourhood of $[s,x]$. We show that it is homeomorphic to $\ran(s, X_{s^*})$ via $\ran|_{(s, X_{s^*})}$. Since $\theta_s\colon X_{s^*}\to X_{s^+}$  and $\dom|_{(s, X_{s^*})}\colon (s, X_{s^*})\to X_{s^*}$ are homeomorphisms, the map $\theta_s\dom|_{(s, X_{s^*})} \colon (s, X_{s^*})\to X_{s^+}$ is also a homeomorphism. This map takes $[s,x]\in (s, X_{s^*})$ to $\theta_s(x) = \ran([s,x])$, so that $\ran|_{(s, X_{s^*})}$ is a homeomorphism, as desired. The remaining claims follow applying \cite[Propositions 7.6 and 7.10]{Kud25}.
\end{proof}

By Proposition \ref{prop:morphism}, the map $\Theta\colon S\to \mC^{op}$ is a $(\cdot\,,^*)$-morphism. We show that for range and birestriction semigroups it also preserves the  $^+$ operation.

\begin{proposition} \label{prop:morphism_range} \mbox{}
\begin{enumerate}
\item Let $S$ be a range semigroup and $\theta\colon S\to \PT_{o}(X)$ be an action. Let, further, $\mC = S\ltimes X$ be the associated category of germs. Then the $(\cdot\,, ^*)$-morphism $\Theta\colon S\to \mC^{op}$, $s\mapsto (s, X_{s^*})$ preserves the  $^+$ operation and is thus a $(\cdot\,, ^*,^+)$-morphism. 
\item Let $S$ be a birestriction semigroup and $\theta\colon S\to \I(X)$ be an action by local homeomorphisms. Let, further, $\mC = S\ltimes X$ be the associated category of germs. Then $\Theta\colon S\to \widetilde{\mC}^{op}$, $s\mapsto (s, X_{s^*})$ is a $(\cdot\,, ^*,^+)$-morphism. 
\end{enumerate}
\end{proposition}

\begin{proof} (1) Applying the definition of $\Theta$, we write:  
\begin{align*}
\Theta(s)^+=(s, X_{s^*})^+ & = \ran(s,X_{s^*}) \\
& = \{\ran([s,x]) \colon x\in X_{s^*}\}  \\
& = \{[s^+, \theta_s(x)] \colon x\in X_{s^*}\} \\
& = (s^+, X_{s^+}) = \Theta(s^+),
\end{align*}
as required.

The second part follows from the first part.
\end{proof}

\section{The universal category and the universal Booleanization}\label{s:bool}
\subsection{Stone duality for Boolean restriction semigroups with local units} \label{subs:duality} For the reader's convenience, we summarise some of the definitions and results from \cite{Kud25} that we will need later on. The following is a special case of \cite[Definition 5.13]{Kud25}.

\begin{definition}\label{def:morphisms1} (Morphisms between Boolean restriction semigroups) Let $S$ and $T$ be Boolean restriction semigroups. 
A map $f\colon S\to T$ will be called a {\em morphism}, if:
\begin{enumerate}
\item $f$ is a $(\cdot\,,^*)$-morphism;
\item the restriction of $f$ to $P(S)$ is a morphism $f|_{P(S)}\colon P(S)\to P(T)$ between generalized Boolean algebras.\footnote{We remind the reader that a morphism between generalized Boolean algebras is assumed to be non-degenerate, see Section \ref{subs:Stone1}.}
\end{enumerate}
\end{definition}

Boolean restriction semigroups with the above-defined morphisms form a category, ${\bf{BRS}}$.

To define morphisms between ample categories, we first recall the definition of an action of a topological category on a topological space introduced in \cite[Definition 7.14]{Kud25}.

\begin{definition} (Action of a topological category on a topological space)
Let $\mC$ be a topological category and $X$ a topological space. 
An {\em action} of $\mC$ on $X$ is a pair $(\mu, f)$, where $f\colon X\to \mC^{(0)}$ is a continuous map and 
$$\mu\colon \mC\times_{\dom,f} X = \{(s,x)\in \mC\times X\colon \dom(s)= f(x)\} \to X,  \quad (s,x)\mapsto s\cdot x
$$ is a continuous map such that the following conditions hold:
\begin{enumerate}
\item[(A1)] $f(s\cdot x) = \ran(s)$ for all $(s,x)\in \mC\times_{\dom,f} X$;
\item[(A2)] $s\cdot (t\cdot x) = st\cdot x$ for all $s,t\in \mC$ and $x\in X$ for which both sides are defined;
\item[(A3)] $f(x)\cdot x = x$ for all $x\in X$.
\end{enumerate}
\end{definition}

If an action $(\mu, f)$ of $\mC$ on $X$ is given,  the {\em transformation category} 
$\mC \ltimes X$ is defined to be the set $\mC\times_{d,f} X$ with $(\mC \ltimes X)^{(0)} = \{(f(x),x)\colon x\in X\}$. The structure maps $\dom, \ran$ and $\pr$ are given by 
$$\dom(s,x) = (f(x),x), \quad \ran(s,x) = (f(s\cdot x), s\cdot x), \quad \pr((s,t\cdot x), (t,x)) = (st,x).
$$ 

The following definition was introduced in \cite[Definition 7.15]{Kud25} and is inspired by \cite{BEM12}.

\begin{definition} (Cofunctors)  A {\em cofunctor} $F=(\mu, f, \rho)\colon \mC \rightsquigarrow {\mathcal D}$ between ample topological categories $\mC$ and ${\mathcal D}$ is given by an action $(\mu, f)$ of $\mC$ on ${\mathcal{D}}^{(0)}$, where $f$ is non-degenerate, and a functor $\rho\colon \mC\ltimes {\mathcal D}^{(0)} \to {\mathcal D}$ between topological categories, which acts identically on units, such that, for every compact slice $A\subseteq \mC$, the set $F_*(A) = \{\rho(s,x)\colon s\in A, (s,x)\in \mC\ltimes {\mathcal D}^{(0)}\}$ is compact slice. 
\end{definition}

Ample topological categories form a category ${\bf ATC}$ whose morphisms are cofunctors.

Let $S$ be Boolean restriction semigroup. In \cite{Kud25}, the author constructed the ample category ${\mathcal C}(S)$ of {\em germs of $S$ over the prime spectrum} $\Spec(P(S))$ {\em of} $P(S)$.\footnote{In \cite{Kud25}, $\Spec(P(S))$ is denoted by $\widehat{P(S)}$. In this paper, $\widehat{P(S)}$ has a different meaning and denotes the character space of the {\em meet-semilattice} $P(S)$.} The construction uses the action of $S$ on $\Spec(P(S))$ given by: 
$$
(s\cdot \varphi)(e) = \varphi((es)^*), \quad e\in P(S), \quad  \varphi \in \Spec(P(S)), \quad s\in S,
$$
in a similar way as this is done in Section \ref{subs:germs}. This gives rise to the functor 
$$
{\mathsf C}\colon {\bf{BRS}} \to {\bf{ATC}}.
$$

Conversely, if ${\mathcal C}$ is an ample category, one can consider the Boolean restriction semigroup ${\mathcal C}^a$ of compact slices of ${\mathcal C}$, which gives rise to the functor
$$
{\mathsf S}\colon  {\bf{ATC}} \to {\bf{BRS}}.
$$
It is proved in \cite[Theorem 9.3]{Kud25} that these functors establish an equivalence between the categories ${\bf{ATC}}$ and ${\bf{BRS}}$. A Boolean restriction semigroup $S$ is naturally isomorphic to ${\mathcal C}(S)^a$ via the isomorphism 
$$
S \to {\mathcal C}(S)^a, \quad s\mapsto (s, D_{s^*}),
$$
where $D_{s^*} = \{\varphi\in \Spec(P(S))\colon \varphi(s^*)=1\}$.
On the other hand, an ample topological category $\mC$ is isomorphic to the category of germs of the Boolean restriction semigroup $\mC^a$ over the prime spectrum $\Spec(P(\mC^a))$ of $P(\mC^a)$. Further details can be found in \cite{Kud25}. 

\subsection{The universal category and the universal Booleanization}
In this section we define the universal category and the universal Booleanization of a restriction semigroup with local units, of a range semigroup and of a birestriction semigroup, which all generalize respective notions for an inverse semigroup, considered, e.g., in \cite{Paterson, St10,L20,K21}. 

\begin{definition} \label{def:univ_categ} (The universal category and the universal Booleanization) 
\begin{enumerate}
\item Let $S$ be a restriction semigroup with local units. We define the {\em universal category} $\mathscr{C}(S)$ of $S$ to be the category of germs $S\ltimes_{\beta} \widehat{P(S)}$ of the spectral action $\beta$ of $S$, defined in Example \ref{ex:spectral}. The Boolean restriction semigroup $\mathscr{C}(S)^a$ of compact slices of the universal category $\mathscr{C}(S)$ will be called the {\em universal Booleanization} of $S$ and will be denoted by ${\mathsf{B}}(S)$.
\item Let $S$ be a birestriction semigroup (in particular, an inverse semigroup). Propositions \ref{prop:s16a} and \ref{prop:range27} imply that the category $\mathscr{C}(S)$ is biample. This means that the set $\widetilde{\mathscr{C}(S)}^a$ of compact bislices forms a Boolean birestriction semigroup, which will be called the {\em restricted universal Booleanization} of $S$, denoted by ${\widetilde{\mathsf{B}}}(S)$.
\end{enumerate}
\end{definition}

If $S$ is a range semigroup, Propositions \ref{prop:s16a} and \ref{prop:range27} imply that the category $\mathscr{C}(S)$ is strongly ample, meaning that ${\mathsf{B}}(S)$ is a Boolean range semigroup. If $S$ is an inverse semigroup, it is well known (see, e.g., \cite{St10}) that $\mathscr{C}(S)$ is an \'etale groupoid.  Furthermore, $\widetilde{{\mathsf{B}}}(S)$ is a Boolean inverse semigroup. If $S$ is a birestriction semigroup (or an inverse semigroup), it has both the universal Booleanization ${\mathsf{B}}(S)$ and the restricted universal Booleanization $\widetilde{{\mathsf{B}}}(S)$\footnote{If $S$ is an inverse semigroup, $\widetilde{{\mathsf{B}}}(S)$ is called the {\em universal Booleanization} of $S$ in \cite{L20}.}.

\begin{remark}\label{rem:basis}
The basis of the topology on  $\mathscr{C}(S)$ is given by the compact slices $(s, U)$ where $s\in S$ and $U\subseteq D_{s^*}$ is a basic open set 
$D_{e,f_1,\dots f_k}=D_e\setminus (D_{f_1}\cup \dots \cup D_{f_k})$, where $k\geq 0$, of $\widehat{P(S)}$. By definition, for each $\varphi\in U$, we have $\varphi(e)=1$, so that $\varphi((se)^*) = \varphi(s^*e) = \varphi(e) =1$. This implies that $[s,\varphi] = [se, \varphi]$. It follows that $(s,U) = (se,U)$. Since $(se)^* = e$, we have $U = D_{(se)^*; f_1,\dots f_k}$.
\end{remark}

The {\em underlying category},  ${\mathscr U}(S)$, of a range semigroup $S$ is defined as follows. Its underlying set is $S$, its units are projections of $S$, the domain and the range maps are defined, for $s\in S$, by $\dom(s) = s^*$ and $\ran(s) = s^+$, and the product is given by
$$
s\cdot t = \left\lbrace\begin{array}{ll} st, & \text{if } s^* = t^+ ;\\
\text{undefined,} & \text{otherwise.}\end{array}\right.
$$

The following result generalizes \cite[Lemma 5.16]{St10}. 

\begin{proposition} \label{prop:dense} Let $S$ be a range semigroup. Then ${\mathrm{U}}(S)=\{[s, (s^{*})^{\uparrow}] \colon s\in S\}$ is a dense subcategory of ${\mathscr C}(S)$, and the map ${\mathscr U}(S) \to {\mathrm{U}}(S)$, $s\mapsto [s,(s^*)^{\uparrow}]$ is an isomorphism.
\end{proposition}

\begin{proof} Since $\ran([s, (s^{*})^{\uparrow}]) = (s^+)^{\uparrow}$, the product $[s, (s^*)^{\uparrow}][t, (t^*)^{\uparrow}]$ is defined if and only if $t^+ = s^*$, in which case it equals $[st, (t^*)^{\uparrow}]$. If $[s, (s^*)^{\uparrow}] = [t, (t^*)^{\uparrow}]$ then $s^*=t^*$ and there is $u\leq s,t$ with $u^*\in (s^*)^{\uparrow}$ such that $[s, (s^*)^{\uparrow}] = [u, (s^*)^{\uparrow}]$. But then $u^*\geq s^*$, so that $u=s$ must hold. By, symmetry, we also have $u=t$. Hence, the map $s\mapsto [s, (s^*)^{\uparrow}]$ is injective. It now easily follows that the map $s\mapsto [s,(s^*)^{\uparrow}]$ is an injective functor. We are left to show that ${\mathrm{U}}(S)$ is a dense subcategory of ${\mathscr C}(S)$. Let $[s, \varphi]\in {\mathscr C}(S)$, and let $(s,U)$ be its basic open neighbourhood. In view of Remark \ref{rem:basis}, we can assume that $U = D_{s^*;f_1,\dots, f_k}$, where $k\geq 0$. Then $(s^*)^{\uparrow}\in U$, so that $[s,(s^*)^{\uparrow}] \in (s,U)$. This proves that   ${\mathrm{U}}(S)$ is dense in ${\mathscr C}(S)$.
\end{proof}

 We now demonstrate that $S$ embeds into ${\mathsf{B}}(S)$.
\begin{proposition} \label{prop:embedding} Let $S$ be a restriction semigroup with local units. The map 
$$
\iota\colon S\to \mathsf{B}(S),\quad s\mapsto (s, D_{s^*})
$$
is an injective morphism of restriction semigroups.
In detail, the set
$$
\iota(S) = \{(s, D_{s^*})\colon s\in S\}
$$ is a $(\cdot\,,^*)$-subalgebra of $\mathsf{B}(S)$ with the operations given by
$$
(s,D_{s^*})(t, D_{t^*}) = (st, D_{(st)^*}), \quad (s, D_{s^*})^* = (s^*, D_{s^*}).
$$
Consequently, the map $\iota\colon S\to \iota(S)$, $s\mapsto (s, D_{s^*})$ is a $(\cdot\,,^*)$-isomorphism.
\end{proposition}

\begin{proof} In view of Proposition \ref{prop:morphism}, only injectivity of the map $\iota$ requires proof. Suppose that $(s, D_{s^*}) = (t, D_{t^*})$, then we have $D_{s^*} = D_{t^*}$. Since the characteristic function
$\chi_{(s^*)^{\uparrow}}$ belongs to $D_{s^*}$, it follows that $\chi_{(s^*)^{\uparrow}}\in D_{t^*}$ and $[s, \chi_{(s^*)^{\uparrow}}] = [t, \chi_{(s^*)^{\uparrow}}]$. By the definition of the equality of germs, there is $u\leq s,t$ such that $\chi_{(s^*)^{\uparrow}}(u^*)=1$ or, equivalently, $u^*\geq s^*$. Given that $u\leq s$, we have $u=su^* \geq ss^* = s$, so $u=s$. It follows that $s\leq t$. By symmetry, we also have $t\leq s$, as required. 
\end{proof}

\begin{corollary}
Let $S$ be a range or a birestriction semigroup. Then $\iota\colon S\to \iota(S)$ is a $(\cdot\,,^*, ^+)$-isomorphism.
Furthermore, if $S$ is a birestriction semigroup, then $\iota(S)\subseteq {\widetilde{\mathsf{B}}}(S)$.
\end{corollary}
 
\begin{proof}
The first claim follows from  Proposition \ref{prop:morphism_range}(1). The second claim follows from Proposition \ref{prop:morphism_range}(2) and the fact that the spectral action of a birestriction semigroup is by local homeomorphisms, see Proposition \ref{prop:s16a}.
\end{proof}

\begin{remark} \label{rem:identity} Let $S$ be a restriction semigroup. If it has local units, Proposition \ref{prop:embedding} gives an embedding of $S$ into
${\mathscr C}(S)^a$. Otherwise, $S^1 = S\cup \{1\}$, where $1\not\in S$ is an external identity, is a restriction monoid where $1^*=1$. Since $S^1$ has local units, there is an embedding 
\begin{equation}\label{eq:def_iota15}
\iota\colon S^1\to {\mathscr C}(S^1)^a, \quad s\mapsto (s, D_{s^*}),
\end{equation}
which restricts to an embedding of $S$ into ${\mathscr C}(S^1)^a$.
\end{remark}

The following example shows that the range of the compact slice $\iota(s) = (s, D_{s^*})$ may not be an open set.
\begin{example}\label{ex:range} Let $s\in PT({\mathbb N}_0)$ be the partial map with ${\mathrm{dom}}(s) = \{0\}$ and such that $s(0)=1$.
Define $E$ to be the subsemilattice of the projection semilattice of $PT({\mathbb N}_0)$ consisting of the empty map $\varnothing$, the identity map ${\mathrm{id}}_{\{0\}}$ on $\{0\}$, and the identity maps ${\mathrm{id}}_A$, where $A$ is a cofinite subset of ${\mathbb N}$ (so that $0\not\in A)$ which contains $1$. Then $S = \{s\} \cup E$ is a restriction subsemigroup of  $PT({\mathbb N}_0)$, and it has local units. The semigroup $S$ is not closed with respect to the $^+$ operation on $PT({\mathbb N}_0)$, since ${\mathrm{id}}_{\{1\}} = s^+ \not\in S$. 
We have that $D_{s^*}$ is the set of filters of $E$ which contain $0$, so that $D_{s^*}$ consists of only one element, $\{0\}^{\uparrow} = \{0\}$. Therefore, $(s, D_{s^*}) = \{[s,\{0\}]\}$, so that $\ran(s, D_{s^*}) = \beta_s(D_{s^*})$ is a singleton, $\{{\mathrm F}\}$, where ${\mathrm F}$ is the filter of $E$ consisting of all $e\in E$ which contain $1$. Since there is no smallest element of this filter, it is not principal.
The set $\{{\mathrm F}\}$ is not open in $\widehat{E}$ since any open set is a union of basic open sets, and a basic open set $D_{e;f_1,\dots, f_k}$ (where $k\geq 0$) contains the principal filter $e^{\uparrow}$. We have shown that $\ran(s, D_{s^*})$ is not an open set.
\end{example}

\subsection{The universal property of the universal Booleanization}
The following result provides the universal property of the universal Booleanization of a restriction semigroup with local units.

\begin{theorem} \label{th:univ_main} (Universal property of the universal Booleanization of a restriction semigroup with local units) Let $S$ be a restriction semigroup with local units, $T$ a Boolean restriction semigroup with local units, and $\alpha\colon S\to T$ a $(\cdot\,,^*)$-morphism, such that the morphism $\alpha|_{P(S)}\colon P(S)\to P(T)$ is non-degenerate. Then there is a morphism $\psi\colon {\mathsf{B}}(S)\to T$ of Boolean restriction semigroups such that $\alpha = \psi\iota$.
\end{theorem}

\begin{proof} Using non-commutative Stone duality \cite[Theorem 9.3]{Kud25} (which is briefly reviewed in Section \ref{subs:duality}),  we can assume that $T$ is isomorphic to $\mC^a$,
where $\mC$ is  an ample category. Then $P(T)$ is the generalized Boolean algebra ${\mathcal B}(\mC^{(0)})$ of compact-open subsets of $\mC^{(0)}$.
Now, 
$$\alpha|_{P(S)} \colon P(S)\to {\mathcal B}(\mC^{(0)})
$$
is a non-degenerate morphism of meet-semilattices from $P(S)$ to the generalized Boolean algebra ${\mathcal B}(\mC^{(0)})$, and Proposition \ref{prop:universal_sem} implies that it extends
to a non-degenerate morphism of generalized Boolean algebras $\psi \colon {\mathsf{B}}(P(S))\to {\mathcal B}(\mC^{(0)})$. Stone duality for generalized Boolean algebras (see Theorem \ref{th:duality_sem}) now implies that 
$$
\psi^{-1} \colon \Spec({\mathcal B}(\mC^{(0)})) \to \Spec({\mathsf{B}}(P(S)))
$$ 
is a proper and continuous map. It follows from Theorem \ref{th:duality_sem} that prime filters of the generalized Boolean algebra ${\mathcal B}(\mC^{(0)})$ are of the form 
$${\mathrm F}_x = \{U \in {\mathcal B}(\mC^{(0)})\colon x\in U\}, \quad \text{where} \quad x\in X, 
$$
and $\Spec({\mathcal B}(\mC^{(0)}))$ is homeomorphic to $\mC^{(0)}$ via the map  ${\mathrm F}_x\mapsto x$.
Moreover, if ${\mathrm F}_x$ is a prime filter of the generalized Boolean algebra ${\mathcal B}(\mC^{(0)})$ then 
$$
\psi^{-1}({\mathrm F}_x)\cap P(S) = \alpha|_{P(S)}^{-1}({\mathrm F}_x) = \alpha^{-1}({\mathrm F}_x)\cap P(S).
$$
It follows from Proposition \ref{prop:7o} that the map
$$\Spec({\mathsf{B}}(P(S)))\to \widehat{P(S)}, \quad {\mathrm F}\mapsto {\mathrm F}\cap P(S)
$$
is a homeomorphism. Therefore, the map
$$
\alpha|_{P(S)}^{-1}\colon   \Spec({\mathsf{B}}(P(S))) \to \widehat{P(S)}, \quad {\mathrm F}_x\mapsto \alpha|_{P(S)}^{-1}({\mathrm F}_x) = \alpha^{-1}({\mathrm F}_x)\cap P(S)
$$
is a proper and continuous map.  In view of \cite[Theorem 9.3]{Kud25}, it suffices to prove that there is a cofunctor 
$${\bf F} = (\mu, f,\rho)\colon {\mathscr{C}}(S) \rightsquigarrow {\mathcal{C}}.$$
It is convenient to identify ${\mathscr C}(S)^{(0)}$ with $\widehat{P(S)}$ via the isomorphism of Lemma \ref{lem:homeo}. Let the map $f\colon \mC^{(0)} \to \widehat{P(S)}$ be defined by
\begin{equation}\label{eq:h1}
f(x) = \alpha|_{P(S)}^{-1}({\mathrm F}_x).
\end{equation}
We now define the action $(\mu,f)$ of ${\mathscr{C}}(S)$ on ${\mathcal C}^{(0)}$. Let $x\in \mC^{(0)}$ and $s\in S$ be such that $\alpha(s^*) = \alpha(s)^* \in {\mathrm F}_x$. Then $[s, f(x)] \in  {\mathscr{C}}(S)$, and we set
\begin{equation}\label{eq:aux5s4}
[s, f(x)] \star x= \mu([s, f(x)],x) = \ran(\alpha(s)x).
\end{equation}
Note that $\alpha(s)$ is a compact slice of $\mC$ and $x\in \dom(\alpha(s)) = \alpha(s)^*$. Therefore, $\alpha(s)x$ is the only $y\in \alpha(s)$ satisfying $\dom(y) = x$.
We verify that $\mu$ is indeed an action. 
We first observe that
\begin{align*}
\alpha(e) \in {\mathrm F}_{\ran(\alpha(s)x)} & \Leftrightarrow  \ran(\alpha(s)x) \in \alpha(e)  \\
& \Leftrightarrow \alpha(s)x = \ran(\alpha(s)x)\alpha(s)x \in \alpha(e)\alpha(s) = \alpha(es).
\end{align*}

For (A1), we need to show that 
$$f(\ran(\alpha(s)x)) = \ran([s, f(x)]).$$
Using the above calculation and \eqref{eq:h1}, we have:
\begin{align*}
e\in f(\ran(\alpha(s)x)) & \Leftrightarrow \alpha(e) \in {\mathrm F}_{\ran(\alpha(s)x)}\\
& \Leftrightarrow \alpha(s)x \in \alpha(es)\\
& \Leftrightarrow x\in \alpha(es)^* \Leftrightarrow (es)^* \in f(x) \\
& \Leftrightarrow e\in \beta_s(f(x)) = \ran([s, f(x)]),
\end{align*}
as desired.
For (A2), we need to show that
$$
([t, \beta_s(f(x))][s,f(x)])\star x = [t, \beta_s(f(x))]\star([s,f(x)] \star x).
$$
The left-hand side equals
$$
[ts, f(x)] \star x = \ran(\alpha(st)x),
$$
while the right-hand side simplifies to
$$
[t, \beta_s(f(x))]\star \ran(\alpha(s)x) = \ran(\alpha(t)\alpha(s)x).
$$
Since $\alpha$ is a morphism, (A2) follows.

(A3) follows from the calculation for $x\in \mC^{(0)}$:
$$
f(x) \star x = [e, f(x)] \star x = \ran(\alpha(e)x) = \ran(x) = x.
$$
 
Let us show that $\mu$ is continuous. Suppose that $[s,f(x)]\star x \in A$ where $A$ is an open set in $\mC^{(0)}$. Then there is a basic open set $\alpha(e)\setminus(\alpha(f_1) \cup \dots \cup \alpha(f_n))$ of $\mC^{(0)}$ such that 
$$[s,f(x)]\star x \in \alpha(e)\setminus (\alpha(f_1)\cup\dots \cup\alpha(f_n)) \subseteq A.$$
 It suffices to show that
there is a neighbourhood, $N$, of $([s,f(x)],x)$ such that $\mu(N)\subseteq \alpha(e)$ and $\mu(N)\cap \alpha(f_i)=\varnothing$ for all $i$. We calculate:
\begin{align*}
[s,f(x)]\star x \in \alpha(e) & \Leftrightarrow  x \in \alpha(es)^*\\
& \Leftrightarrow \alpha((es)^*) \in {\mathrm F}_x\\
& \Leftrightarrow (es)^*\in \alpha|_{P(S)}^{-1}({\mathrm F}_x)\\
& \Leftrightarrow (es)^* \in f(x)\\
& \Leftrightarrow f(x) \in D_{(es)^*}.
\end{align*}
Setting $$B=  \alpha((es)^*)\setminus (\alpha((f_1s)^*)\cup \dots \cup \alpha((f_ns)^*))$$ and 
$$N = \{([s, f(x)], x)\colon x\in B\},$$
we have that
$$N = ((s, D_{(es)^*})\times B) \cap ({\mathscr{C}}(S) \ltimes {\mathcal C}^{(0)}).$$ 
Hence $N$ is an open set and $\mu(N)\subseteq \alpha(e)\setminus (\alpha(f_1)\cup\dots \cup\alpha(f_n))$, as desired.

We finally define $\rho \colon {\mathscr{C}}(S) \ltimes {\mathcal C}^{(0)} \to {\mathcal C}$ by
$$
\rho([s, f(x)],x) = \alpha(s)x.
$$
The routine verification that $\rho$ is a functor between topological categories, which is identical on units, is omitted. We show that ${\bf F}_*$ maps compact slices onto compact slices.
By the definition of the product topology on ${\mathscr{C}}(S) \ltimes {\mathcal C}^{(0)}$ and bearing in mind Remark \ref{rem:basis}, it suffices to show that the set 
\begin{equation}\label{eq:14o1}
\rho\{([s,f(x)], x) \colon x\in \alpha(s^*)\setminus (\alpha(f_1)\cup\dots \cup\alpha(f_n))\}
\end{equation}
is a compact slice, where $n\geq 0$ and $f_i\leq s^*$ for all $i$. But this is clearly so, since the set in \eqref{eq:14o1} equals
$$\{\alpha(s)x \colon x\in \alpha(s^*)\setminus (\alpha(f_1)\cup\dots \cup\alpha(f_n))\} = \alpha(s)\setminus (\alpha(sf_1)\cup\cdots \cup \alpha(sf_n)).
$$
Since the latter set coincides with $\alpha(s)(\alpha(s^*)\setminus (\alpha(f_1)\cup\dots \cup\alpha(f_n))$, it belongs to  ${\mathscr C}^a$. This completes the proof.
\end{proof}

For range semigroups, we have the following similar result.

\begin{theorem} (Universal property of the universal Booleanization of a range semigroup) \label{th:univ_bool} Let $S$ be a range semigroup, $T$ a Boolean range semigroup, and $\alpha\colon S\to T$ a $(\cdot\,,^*,^+)$-morphism, such that the morphism $\alpha|_{P(S)}\colon P(S)\to P(T)$ is non-degenerate. Then there is a morphism $\psi\colon {\mathsf{B}}(S)\to T$ of Boolean range semigroups such that $\alpha = \psi\iota$.
\end{theorem}

\begin{proof} By Theorem \ref{th:univ_main}, there is a $(\cdot\,, ^*)$-morphism $\psi\colon {\mathsf{B}}(S)\to T$ such that $\alpha = \psi\iota$. We show that $\psi$ also preserves the $^+$ operation. 
For this, we verify that the cofunctor ${\bf F} = (\mu, f,\rho)\colon {\mathscr{C}}(S) \rightsquigarrow {\mathcal{C}}$, constructed in the previous proof, has the property that ${\bf F}_*$ preserves the ranges of compact slices. Let $A=(s,U)$ be a compact slice, where $s\in S$ and $U\subseteq D_{s^*}$ is a compact-open set. Since ${\bf F}_*(A) = \{\rho(s,x)\colon s\in A, (s,x)\in {\mathscr C}(S) \ltimes \mC^{(0)}\}$, we have:
$$
{\bf F}_*(s, U) = \{\rho([s,f(x)],x) \colon f(x) \in U\} = \{\alpha(s)x \colon x\in f^{-1}(U)\}.
$$
Therefore,
$$\ran({\bf F}_*(s, U)) = \{\ran(\alpha(s)x) \colon x\in f^{-1}(U)\},
$$
which is compact-open, since $f^{-1}$ is proper and continuous, and $\ran$ and $\pr$ are continuous (so they take compact sets to compact sets, as we work with Hausdorff spaces).
\end{proof}

Let now $S$ be a birestriction semigroup. Propositions \ref{prop:s16a} and \ref{prop:range27} imply that the category $\mC$ is in this case biample.  Applying \cite[Theorems 9.6 and 9.8]{Kud25}, one can readily adapt the above proof to show the following. 

\begin{theorem} (Universal property of the restricted universal Booleanization of a birestriction or an inverse semigroup)
Let $S$ be a birestriction semigroup (resp. an inverse semigroup), $T$ a Boolean birestriction semigroup (resp. a Boolean inverse semigroup), and $\alpha\colon S\to T$ a $(\cdot\,,^*,^+)$-morphism (resp. a semigroup morphism), such that the morphism $\alpha|_{P(S)}\colon P(S)\to P(T)$ is non-degenerate. Then there is a mor\-phism $\psi\colon \widetilde{{\mathsf{B}}}(S)\to T$ of Boolean birestriction  semigroups (resp. of Boolean inverse semigroups) such that $\alpha = \psi\iota$.
\end{theorem}

A variant of the above result for the universal property of $\widetilde{{\mathsf{B}}}(S)$ for an inverse semigroup with zero $S$ was proved in \cite{L20} using different techniques.

\subsection{The finitary case} Suppose that $S$ is a restriction semigroup with local units such that each principal downset of $P(S)$ is finite. Then every filter of $P(S)$ is principal, so that $\widehat{P(S)}$ is in a bijection with $P(S)$ (via $e^{\uparrow}\mapsto e$). In the special case when $S$ is an inverse semigroup, it is easy to verify and well known that the category ${\mathscr C}(S)$ is a groupoid and coincides with the underlying groupoid of $S$ (see, e.g., \cite[Definition 5.14]{St10}). More generally, if $S$ is a range semigroup then  ${\mathscr C}(S)$ similarly coincides with the underlying category of $S$. In the general case, we have the following result.

\begin{proposition} \label{prop:finitary} Let $S$ be a restriction semigroup with local units such that each principal downset of $P(S)$ is finite. The following statements hold.
\begin{enumerate}
\item
${\mathscr C}(S) = \{[s, (s^{*})^{\uparrow}]\colon s\in S\}$
and $s\to [s, (s^{*})^{\uparrow}]$ is a bijection between $S$ and ${\mathscr C}(S)$.
\item Let $\beta$ be the spectral action of $S$, $s\in S$ and $e\in {\mathrm{dom}}(\beta_s)$. Then 
$\beta_s({e^{\uparrow}}) = {f^{\uparrow}}$ where $f$ is the smallest projection $h$ such that $(hs)^*\geq e$.
\item Suppose that $S$ is a range semigroup. Then ${\mathscr C}(S)$ is isomorphic to the underlying category ${\mathscr U}(S)$ of $S$ (as a discrete category).
\end{enumerate}
\end{proposition}

\begin{proof}
(1) By definition, the elements of ${\mathscr C}(S)$ are germs $[s,e^{\uparrow}]$ such that $s^*\in e^{\uparrow}$, which means $s^{*} \geq e$. For such a projection $e$, we have $(se)^* = s^*e=e$, so that  $[s,e^{\uparrow}] = [se, ((se)^*)^{\uparrow}]$. Thus, the elements of ${\mathscr C}(S)$ are germs $[s, (s^*)^{\uparrow}]$ where $s\in S$. To show that the map $s\mapsto [s,(s^*)^{\uparrow}]$ is injective, suppose $[s,(s^*)^{\uparrow}] = [t,(s^*)^{\uparrow}]$. The definition of the equality of germs implies that $s^* = t^*$, and there is $u\leq s,t$ such that $u^*\geq s^*$. Therefore, $s=t$. Since the map $s\to [s, (s^{*})^{\uparrow}]$ is clearly surjective, it is a bijection.

(2) Since $e^{\uparrow} \in {\mathrm{dom}}(\beta_s)$, we have $s^*\geq e$. Then $(se)^* = e$ and $\beta_s(e) = \beta_{se}(e)$. We know that $\beta_s(e^{\uparrow})$ is a principal filter, denote it by $f^{\uparrow}$. By the definition of $\beta_s$, we have $h\in \beta_s(e^{\uparrow}) \Leftrightarrow (hs)^* \in e^{\uparrow}$, which rewrites to $h \geq f \Leftrightarrow (hs)^* \geq e$. It follows that $f$ is the minimum projection $h$, such that $(hs)^* \geq e$, as desired.

(4) The statement follows from Proposition \ref{prop:dense} and part (1).
\end{proof}

Let $S$ be a restriction semigroup with local units such that each principal downset of $P(S)$ is finite. In view of part (2) of Proposition \ref{prop:finitary}, one can define the `range' $s^{\oplus}$ of $s\in S$ as the range of the arrow $[s, (s^*)^{\uparrow}]$ of the category ${\mathscr C}(S)$, that is, to define $s^{\oplus}$ to be the generator of the principal filter $\beta_s((s^*)^{\uparrow})$. This is precisely the minimum projection $e$ such that $(es)^* \geq s^*$. The following example shows that $(S; \,\cdot, ^*,\, ^{\oplus})$ is not, in general, a range semigroup.

\begin{example} Let $X=\{a,b\}$, $f$ be the identity map on $\{a\}$, and $g$ be the map with ${\mathrm{dom}}(g) = a$ such that $g(a)=b$.  Let $S$ be the restriction submonoid of $PT(X)$ generated by $f$ and $g$. Then $S=\{f, g, \varnothing, 1\}$, where $\varnothing$ is the empty map and $1$ is the identity map on $X$. By the definition of the $^*$ operation in $PT(X)$, we have that
$f^* = g^* = f$, $\varnothing^* = \varnothing$ and $1^* = 1$. Note that $S$ is not a $(\cdot\,,^*, ^+)$-submonoid of $PT(X)$ as $g^+$ is the identity map on the set $\{2\}$, which does not belong to $S$. It is readily verified that $\varnothing^{\oplus} = \varnothing$, $f^{\oplus} = f$, $g^{\oplus} = 1$ and $1^{\oplus} = 1$. However, we have:
$(fg)^{\oplus} = \varnothing^{\oplus} = \varnothing \neq f = f^{\oplus} = (fg^{\oplus})^{\oplus}$.
Therefore, the third axiom of \eqref{eq:axioms_plus} is not satisfied for $(S; \,\cdot, ^*,\, ^{\oplus})$, and thus it is not in  a range semigroup.
\end{example}

We conclude this section by recalling that for the restriction semigroup $S$ and its element $s$ from Example \ref{ex:range} the filter $\beta_s ((s^*)^{\uparrow})$  not a principal. This demonstrates that in the general case (with no restrictions imposed on $P(S)$) $\beta_s({\mathrm F})$ does not need to be a principal filter when   ${\mathrm F} \in \widehat{P(S)}$ is a principal filter, so that the operation $s\mapsto s^{\oplus}$ can not be extended to arbitrary restriction semigroups with local units.

\section{One can see ranges of inductive constellations}\label{s:see_ranges}
\subsection{Inductive constellations} 
Constellations were introduced in \cite{GH09} as analogues of categories, where arrows have domains but do not have ranges. The ESN-type theorem due to Gould and Hollings~\cite{GH09} states that the category of restriction semigroups is isomorphic to the category of inductive constellations. In this section, we explain that every inductive constellation admits a topological representation as an inductive constellations of compact slices of an ample category. 

We begin with the definition of a constellation given in  \cite{GS17, GS22}\footnote{Beware that we use the dual notion to that used in \cite{GS17, GS22}, since we multiply from the right to the left.}.
If $(Q, \cdot)$ is a set with a partial binary operation $\cdot$ then $e\in Q$ is a {\em right identity} for $x\in Q$ if $x\cdot e$ is defined and $x\cdot e=x$. 

\begin{definition} (Constellations)
A  (small) {\em constellation} is a set $Q$ equipped with a partial binary operation $\cdot$ satisfying the following:
\begin{enumerate}
\item[(Q1)] if $(s\cdot t)\cdot u$ exists then $s\cdot (t\cdot u)$ exists, and then $s\cdot (t\cdot u) = (s\cdot t)\cdot u$;
 \item[(Q2)] if $s\cdot t$ and $t\cdot u$ exist then so does $(s\cdot t)\cdot u$;
\item[(Q3)] for each $s\in Q$, there is a unique right identity, denoted by $s^*$.
\item[(Q4)] for each $s,t\in Q$, if there exists $s^*\cdot t$ then $s^* \cdot t =t$. 
\end{enumerate} 
\end{definition}

The set $P(Q)=\{s^* \colon s\in Q\}$ is called the set of {\em projections} of $Q$. Since $^*$ can be viewed as a unary operation, constellations can be viewed as partial algebras $(Q; \,\cdot, ^*)$.
The following definition combines \cite[Definitions 3.1 and 3.3]{GH09}.

\begin{definition} (Inductive constellations) Let $(Q; \,\cdot, \, ^*)$ be a constellation and $\leq$ a partial order on $Q$. We say that $(Q; \, \cdot, \, ^*, \leq)$ is an {\em ordered constellation} if:
\begin{enumerate}
\item[(O1)] If $s\leq t$, $u\leq v$ and $s\cdot u$, $t\cdot v$ exist, then $s\cdot u\leq t\cdot v$;
\item[(O2)] If $s\leq t$ then $s^* \leq t^*$;
\item[(O3)] If $e\in P(Q)$ and $s\in Q$ are such that $e\leq s^*$ then there is the {\em restriction} $s|_e$ which is the unique element with the properties $s|_e \leq s$ and $(s|_e)^*=e$;
\item[(O4)] for all $e\in P(Q)$ and $s\in Q$ there exists a {\em corestriction} ${}_{e}|s$, which is the maximum element $t$ such that $t\leq s$ and the product $e\cdot t$ exists.
\item[(O5)] If $s,t\in Q$ are such that there exists $s\cdot t$ and $e\in P(Q)$, then $({}_e|(s\cdot t))^* = ({}_{({}_e|s)^*}|t)^*$.
\item[(O6)] If $e,f \in P(Q)$ and $e|_f$ is defined, then  $e|_f = {}_f|e$.
\end{enumerate} 
An ordered constellation is called {\em inductive} if the following condition holds:
\begin{enumerate}
\item[(I)] $P(Q)$ is a semilattice with $e\wedge f = {}_f|e$.
\end{enumerate}
We sometimes abbreviate $(Q; \, \cdot, ^*, \leq)$ simply by $Q$.
\end{definition}

The next definition combines \cite[Definitions 2.5, 3.8]{GH09}.
\begin{definition} (Ordered radiants and embeddings)
Let $Q$ and $R$ be ordered constellations. A function $\rho\colon Q\to R$ is called an {\em ordered radiant}, if the following conditions hold:
\begin{enumerate}
\item If $s\cdot t$ exists in $Q$ then $\rho(s)\cdot \rho(t)$ exists in $R$, in which case $\rho(s\cdot t) = \rho(s)\cdot \rho(t)$.
\item For all $s\in Q$: $\rho(s^*) = \rho(s)^*$.
\item If $s\leq t$ in $Q$ then $\rho(s)\leq \rho(t)$ in $R$.
\item For all $s\in Q$ and $e\in P(Q)$: $\rho({}_e|s) = {}_{\rho(e)}|\rho(s)$.
\end{enumerate}
We say that an ordered radiant $\rho$ is {\em strong} if (1) and (3) are strengthened to:
\begin{enumerate}
\item[(1')] $s\cdot t$ exists in $Q$ if and only if $\rho(s)\cdot \rho(t)$ exists in $R$; in which case $\rho(s\cdot t) = \rho(s)\cdot \rho(t)$.
\item[(3')]  $s\leq t$ in $Q$ if and only if $\rho(s)\leq \rho(t)$ in $R$.
\end{enumerate}
Note that (3') implies that $\rho$ is necessarily injective, thus a strong ordered radiant is also called an {\em embedding}. 
\end{definition}

\begin{definition} (Isomorphism) Let $Q$ and $R$ be constellations. A function $\rho\colon Q\to R$ will be called an {\em isomorphism}, if it is bijective and is an embedding.
\end{definition}

Let $S$ be a restriction semigroup, whose multiplication is denoted by juxtaposition and natural partial order by $\leq$. Define a partial product $\circ$ on $S$ by
\begin{equation}\label{eq:partial_pr}
s \circ t = \left\lbrace\begin{array}{ll} st, &  \text{ if } s^*t = t;\\
\text{undefined}, & \text{ otherwise.}\end{array}\right..
\end{equation}
It was shown in \cite[Proposition 4.1]{GH09} that $(S; \, \circ, ^*, \leq)$ is an inductive constellation, denoted by ${\bf P}(S)$. It will be convenient to call ${\bf P}(S)$ the inductive constellation {\em associated to} $S$. We note that if $S$ is a range semigroup, $s\circ t$ is defined if and only if $s^*\leq t^+$.

The other way around, let $(Q;\, \cdot, ^*, \leq)$ be an inductive constellation. 
Observe that for $s,t\in Q$ the product $s \cdot {}_{s^*}|t$ exists. Indeed,  $s^*\cdot {}_{s^*}|t$ exists by (O4), $s\cdot s^*$ exists and equals $s$ by (Q3). Thus, in view of (Q1) and (Q2), $(s\cdot s^*)\cdot {}_{s^*}|t = s\cdot (s^* \cdot {}_{s^*}|t) = s\cdot {}_{s^*}|t$ is defined. We can thus define the binary operation $\otimes$ on $Q$, called the {\em pseudoproduct}, by the rule:
$$
s\otimes t = s \cdot {}_{s^*}|t.
$$
It is proved in \cite[Proposition 4.5]{GH09} that $(Q; \, \otimes, \, ^*)$ is a restriction semigroup. It is denoted by ${\bf T}(Q)$.
Furthermore, for an inductive constellation $Q$, we have ${\bf P}({\bf T}(Q)) = Q$, and for a restriction semigroup $S$, we have ${\bf T}({\bf P}(S)) = S$. This gives rise to the following ESN-type theorem due to Gould and Hollings \cite[Theorem 4.13]{GH09}.

\begin{theorem} \label{th:s15a} The category of restriction semigroups with $(\cdot\,,^*)$-morphisms is isomorphic to the category of inductive constellations with ordered radiants as morphisms.
\end{theorem}

As a consequence, we have the following.

\begin{proposition} \label{prop:isom_const} Let $\rho \colon S\to T$ be an embedding of restriction semigroups. We define ${\bf P}(\rho) \colon {\bf P}(S) \to {\bf P}(T)$ to be the same function on the underlying sets. Then ${\bf P}(\rho)$ is an embedding of inductive constellations. If $\rho$ is an isomorphism, then so is  ${\bf P}(\rho)$.
\end{proposition}

\subsection{Inductive constellations of compact slices of ample categories}
The following easily follows from the definitions.

\begin{proposition} \label{prop:const_slices} Let ${\mathcal C}$ be an ample category and $S$ be a restriction subsemigroup of the restriction semigroup ${\mathcal C}^a$. Let $({\bf P}(S);\,\circ \,, ^*, \leq)$ be the associated inductive constellation. Then: 
\begin{enumerate}
\item The partial product \eqref{eq:partial_pr} is given by
 \begin{equation}\label{eq:partial_pr11}
s \circ t = \left\lbrace\begin{array}{ll} st, &  \text{ if } \dom(s) \supseteq \ran(t);\\
\text{undefined}, & \text{ otherwise.}\end{array}\right.
\end{equation}
\item The partial order is given by subset inclusion, that is, $s\leq t$ if and only if $s \subseteq t$.
\item If $e\in P(S)$ and $s\in S$ are such that $e\leq s^*$, the restriction $s|_e$ is given by
$$
s|_e = se = \{x\in s\colon \dom(x) \in e\}.
$$
\item If $e\in P(S)$ and $s\in S$, the corestriction ${}_{e}|s$ is given by
$$
{}_{e}|s = es = \{x\in s\colon \ran(x)\in e\}.
$$
\end{enumerate}
\end{proposition}

\begin{definition}
We say that the inductive constellation that arises in Proposition \ref{prop:const_slices} is an {\em inductive constellation of compact slices of an ample category.} 
\end{definition}

We say that a constellation $Q$ has {\em local units} if for every $s\in Q$ there is $e\in P(Q)$ such that ${}_e|s=s$, which implies that $e\cdot s$ exists and equals $s$. It is immediate that $Q$ has local units if and only if the restriction semigroup ${\bf T}(Q)$ has local units. Proposition \ref{prop:embedding}  together with Proposition \ref{prop:isom_const} and the equality ${\bf P}({\bf T}(Q))=Q$ imply the following statement. 

\begin{proposition}\label{prop:main15a}
Let $Q$ be an inductive constellation with local units and $\iota$ be the isomorphism ${\bf T}(Q) \to \iota({\bf T}(Q))$ of Proposition \ref{prop:embedding}. Then 
$$
{\bf P}(\iota) \colon Q= {\bf P}({\bf T}(Q)) \to  {\bf P}(\iota({\bf T}(Q))) 
$$
is an isomorphism of inductive constellations. \end{proposition}

If $Q$ is an arbitrary inductive constellation, then ${\bf T}(Q)$ embeds into ${\mathscr C}({\bf T}(Q)^1)^a$ by Remark \ref{rem:identity}. This leads to the embedding of $Q={\bf P}({\bf T}(Q))$ into the inductive constellation ${\bf P}({\mathscr C}({\bf T}(Q)^1)^a)$. Therefore, every inductive constellation is isomorphic to an inductive constellation of compact slices of an ample category.
We arrive at the following refinement of the Gould-Hollings ESN-type theorem recalled in Theorem \ref{th:s15a}. 
 
\begin{theorem} \label{th:main15} \mbox{}
\begin{enumerate}
\item The category of inductive constellations is equivalent its to full subcategory of inductive constellations of compact slices of ample categories. 
\item The category of restriction semigroups  is equivalent to the category of inductive constellations of compact slices of ample categories.
\end{enumerate}
\end{theorem}

Theorem \ref{th:main15}(2) may be viewed as a {\em topological variant} of Theorem \ref{th:s15a}. Unlike the latter, this result does not provide an {\em isomorphism}, but rather an {\em equivalence} of categories. On the other hand, it shows that elements of inductive constellations, which have domains but no ranges, can be represented as slices of categories, which have both domains and ranges. This justifies the title of this section.

\begin{remark} Theorem \ref{th:main15} paves the way for a conceptually new (though not simpler, due to the use of topological machinery) proof of Theorem \ref{th:s15a}, which we now outline. Let $Q$ be an inductive constellation with local units\footnote{In the general case, we can carry out the construction for $Q^1$.}. Let us show that $Q$ gives rise to its {\em universal category} ${\mathscr C}(Q)$. The construction of the spectral action of $Q$ on $\widehat{P(Q)}$ and of the category ${\mathscr C}(Q)$ is carried out exactly in the same way as for restriction semigroups, except that the left and right multiplication by a projection are replaced by the left and right restriction to a projection. A subtlety arises only when we come to the definition of the multiplication of germs. In view of \cite[Lemma 3.2]{GH09}, the product $s\cdot t$ exists in $Q$ if and only if $_{s^*}|t = t$. We show that such products suffice to define the product of germs. Let $x\in \widehat{P(Q)}$ and let $s,t\in Q$ be such that $[s,\beta_t(x)]$ and $[t,x]$ are germs. We can not simply define their product as $[s\cdot t,x]$,  as $Q$ is only a partial algebra, and the product $s\cdot t$ may not exist in $Q$. Note that $x$ is in $D_{(_{s^*}|t)^*}$ and $[t, x] = [_{s^*}|t, x]$. Therefore, setting $u= \, _{s^*}|t$, we see that $[t,x] = [u,x]$ and $_{u^*}|t = t$, so that the product $s\cdot u$ exists in $Q$. We can therefore define $[s,\beta_t(x)][t,x] = [s,\beta_u(x)][u,x] = [s\cdot u, x]$. We have the embedding of inductive constellations $\iota\colon Q\to \iota(Q)$, $s\mapsto (s, D_{s^*})$. On the inductive constellation $\iota(Q)$ there is a natural pseudoproduct, which is simply the product of slices. This gives rise to the restriction semigroup ${\bf T}(\iota(Q))$. Pulling this pseudoproduct back via $\iota$ to $Q$, we obtain the restriction semigroup ${\bf T}(Q)$. This shows that the pseudoproduct on ${\bf T}(Q)$ arises naturally.  Since the product of slices is a special case of the product of subsets of a category, the pseudoproduct is automatically associative. The inductive constellation $\iota(Q)$ thus naturally serves as a mediator  between $Q$ and ${\mathbf T}(Q)$ providing a `concrete' realization of the passage from the constellation $Q$ to the restriction semigroup ${\bf T}(Q)$.

In the same vein, if $Q$ is an inductive groupoid, one can define its spectral action on its objects, $E(Q)$, and thus construct the universal groupoid ${\mathscr G}(Q)$. In particular, to define the product $[s,\beta_t(x)][t,x]$ one observes that $[t,x] =  [_{s^*}|t, x]$ and $[s, \beta_t(x)] = [s|_{t^+}, \beta_t(x)]$. Setting $u=\,_{s^*}|t$, $v=s|_{t^+}$ and observing that the product $v\cdot u$ is defined in $Q$, we have $[s,\beta_t(x)][t,x] = [v,\beta_u(x)][u,x] = [v\cdot u,x].$ The pseudoproduct on $\iota(Q)$ again arises naturally as the product of bislices, which similarly leads to a new proof of the ESN theorem for inverse semigroups.
\end{remark}

\subsection{A topological ESN-type theorem for range semigroups}
Let $S$ be a range semigroup. It is well known  that $S$  can be reconstructed from the following data: its underlying category ${\mathscr U}(S)$; the natural partial order $\leq$ on $S$; restrictions and corestrictions, given by 
 $s|_e = se$ and ${}_e|s = es$. The product in $S$ is then reconstructed as 
 $$
 st = s|_{s^*t^+} \cdot {}_{s^*t^+}|t,
 $$
where $s^*t^+$ is the meet  $s^*\wedge t^+$ with respect to the order $\leq$ and $\cdot$ is the product in the category ${\mathscr U}(S)$. It follows from \cite{L91} that the category of ordered categories with restrictions and corestrictions is isomorphic to the category of range semigroups\footnote{For a far-reaching generalization of this result, to ${\mathrm{DRC}}$-semigroups, see \cite{EFAMS24}.}.
By Proposition \ref{prop:embedding} and Proposition \ref{prop:morphism_range} (note that a range semigroup always has local units), there is an $(\cdot\,, ^*, ^+)$-embedding
$$
\iota\colon S\to {\mathscr C}(S)^a, \quad s\mapsto (s, D_{s^*}).
$$
Similarly as before, one can now conclude that the ordered category ${\mathscr U}(S)$ is isomorphic to the ordered category ${\mathscr U}(\iota(S))$. This yields a topological representation of ordered categories with restrictions and corestrictions. The result further specializes to inductive categories and inductive groupoids. Therefore, we obtain topological ESN-type theorems for range, birestriction and inverse semigroups, which are proved similarly as Theorem~\ref{th:main15}. 
\begin{theorem}\mbox{}
\begin{enumerate}
\item The category of range semigroups is equivalent to the category of ordered categories with restrictions and corestrictions of compact slices of strongly ample categories.
\item The category of birestriction semigroups is equivalent to the category of inductive categories of compact bislices of biample categories.
\item The category of inverse semigroups is equivalent to the category of induc\-tive grou\-poids of compact bislices of ample groupoids.
\end{enumerate}
\end{theorem}

\section{Extending the Petrich-Reilly theorem to proper restriction semigroups}\label{s:proper}
Recall from Section \ref{subs:proper} that a restriction semigroup $S$ is called {\em proper} if, for all $s,t\in S$, the conditions $s\mathrel{\sigma} t$ and $s^* = t^*$ imply that $s=t$. 

\subsection{The category of germs of an action of a proper restriction semigroup} Suppose that $S$ is a proper restriction semigroup, and let $T=S/\sigma$. Suppose, further, that $\theta\colon S\to \PT(X)$ is an action. Inspired by \cite[Section 3]{MS14}, in this section we construct a partial action $\overline{\theta}\colon T\to \PT(X)$ such that $S\ltimes_{\theta} X$ is topologically isomorphic to $T\ltimes_{\overline{\theta}} X$.

Recall that for $s\in S$, we denote the $\sigma$-class of $s$  by $[s]_{\sigma}$. For $t\in T$, we define 
$$\sigma^{-1}(t) = \{s\in S\colon [s]_{\sigma} = t\}.
$$
We further define
\begin{equation}\label{def:varphi1}
\mathrm{dom}(\overline{\theta}_t) = \bigcup_{s\in \sigma^{-1}(t)} \mathrm{dom}(\theta_s) = \bigcup\limits_{s\in \sigma^{-1}(t)} X_{s^*}.
\end{equation}
Equivalently, $x\in\mathrm{dom}(\overline{\theta}_t)$ if and only if $x\in \mathrm{dom}(\theta_s)$ for some $s\in \sigma^{-1}(t)$.
To define $\overline{\theta}_t(x)$ for $x\in \mathrm{dom}(\overline{\theta}_t)$, we first prove the following lemma.

\begin{lemma}
Let $t\in T$, $x\in X$ and $u,v\in \sigma^{-1}(t)$ be such that $x\in \mathrm{dom}(\theta_u) \cap \mathrm{dom}(\theta_v)$. Then $[u,x] = [v,x]$.
\end{lemma}

\begin{proof} 
Since $S$ is proper and $u\mathrel{\sigma} v$, Lemma \ref{lem:proper} implies that $u\smile v$, that is, $uv^* = vu^*$. Furthermore, given that $x\in {\mathrm{dom}}(\theta_{v}^*)$ and $x=\theta_{v^*}(x) \in {\mathrm{dom}}(\theta_u)$, we have $x\in {\mathrm{dom}}(\theta_u\theta_{v^*}) = {\mathrm{dom}}(\theta_{uv^*})$.  It follows that $[u,x] = [uv^*,x] = [vu^*,x] = [v,x]$.
\end{proof}

We can now define $\overline{\theta}_t(x)$ for $t\in T$ and $x\in {\mathrm{dom}}(\overline{\theta}_t)$ as 
\begin{equation}\label{def:varphi}
\overline{\theta}_t(x) = \theta_u(x),
\end{equation}
where $u\in \sigma^{-1}(t)$ is any element satisfying $x\in {\mathrm{dom}}(\theta_u)$.

\begin{lemma} For every $t\in T$ we have that $\overline{\theta}_t\in \PT(X)$ and the map $T\to \PT(X)$, $t\mapsto \overline{\theta}_t$ is a premorphism. That is, $T$ acts partially on $X$ via $\overline{\theta}$. 
\end{lemma}

\begin{proof} Since the union of the domains of all $\theta_s$, where $s$ runs through $S$, coincides with $X$, and ${\mathrm{dom}}(\theta_s) = {\mathrm{dom}}(\theta_{s^*})$,  the union of the domains of $\theta_e$, where $e\in P(S)$, coincides with $X$, too. It follows that $\overline{\theta}_1 = {\mathrm{id}}_X$. Suppose that $x\in {\mathrm{dom}}(\overline{\theta}_s)$ and $\overline{\theta}_s(x)\in {\mathrm{dom}}(\overline{\theta}_t)$. Then there are $u\in \sigma^{-1}(s)$ and $v\in \sigma^{-1}(t)$ such that $x\in  {\mathrm{dom}}(\theta_u)$ and $\overline{\theta}_s(x) = \theta_u(x)\in {\mathrm{dom}}(\theta_v) =  {\mathrm{dom}}(\overline{\theta}_t)$. Then $x\in {\mathrm{dom}}(\theta_{u}\theta_v)={\mathrm{dom}}(\theta_{uv})$. Since $[uv]_{\sigma} = [u]_{\sigma}[v]_{\sigma} = st$, we conclude that $x\in {\mathrm{dom}}(\overline{\theta}_{st})$. Furthermore, $\overline{\theta}_{st}(x) = \theta_{uv}(x) = \theta_u\theta_v(x) = \overline{\theta}_s\overline{\theta}_t(x)$. This finishes the proof.
\end{proof}

It follows that we can form the partial transformation category $T\ltimes_{\overline{\theta}} X$. By construction, there is a well-defined map $T\ltimes_{\overline{\theta}} X \to S\ltimes_{\theta} X$ given by
$(t,x) \mapsto [u,x]$ where $u\in \sigma^{-1}(t)$ is such that $x\in {\mathrm{dom}}(\theta_u)$. This map is surjective as $([u]_{\sigma},x)\mapsto [u,x]$ for any $[u,x] \in S\ltimes_{\theta} X$.
Moreover, if $(t,x), (s,x) \mapsto [u,x]$ then $t=[u]_{\sigma} = s$, so that this map is also injective. We arrive at the following generalization to proper restriction semigroups of the result by Milan and Steinberg \cite[Theorem 3.2]{MS14}, which was proved in the setting of $E$-unitary inverse semigroups.

\begin{theorem} \label{th:isom_proper} Let $S$ be a proper restriction semigroup which acts on a locally compact Hausdorff space $X$ via $\theta\colon S\to \PT(X)$. Let $\overline{\theta}$ be the partial action of $T=S/\sigma$ on $X$ defined in \eqref{def:varphi1} and \eqref{def:varphi}. Then the topological categories $T\ltimes_{\overline{\theta}} X$ and $S\ltimes_{\theta} X$ are isomorphic.
\end{theorem}
\begin{proof} We show that the map $f\colon T\ltimes_{\overline{\theta}} X \to S\ltimes_{\theta} X$ given by
$f(t,x) = [u,x]$, where $u\in \sigma^{-1}(t)$ is such that $x\in {\mathrm{dom}}(\theta_u)$, is an isomorphism of topological categories. It is immediate by the definition of $f$ that it commutes with $\dom$ and $\ran$ and preserves the product. Bearing in mind that it is bijective, it is an isomorphism. It remains to show that $f$ and $f^{-1}$ are continuous. Let $(s,Y) = \{[s,y]\colon y\in Y\}$ be a basic open set of $S\ltimes_{\theta} X$. Then $f^{-1}(s,Y) = ([s]_{\sigma}, Y)$, which is open in  $T\ltimes_{\overline{\theta}} X$. Conversely, let $(s,Y)$ be a basic open set in $T\ltimes_{\overline{\theta}} X$. Then 
$$f(s, Y) = \bigcup_{u\in \sigma^{-1}(s)} (u, X_{u^*}\cap Y)$$
which is also open. This finishes the proof. 
\end{proof}

We record the following immediate consequence of Theorem \ref{th:isom_proper}.

\begin{corollary} \label{cor:isom1} Let $S$ be a proper restriction semigroup with local units. Then the universal category ${\mathscr{C}}(S)$ is topologically isomorphic to the partial transformation category
\mbox{$S/\sigma \ltimes_{\overline{\beta}} \widehat{P(S)}$.}
\end{corollary}

Since the partial action $\overline{\beta}$ will play an important role in the sequel, we now record its explicit description, which follows from the definition of the spectral action $\beta$, combined with \eqref{def:varphi1} and \eqref{def:varphi}. For any $t\in S/\sigma$, the domain of $\overline{\beta}_t$ is the set

\begin{equation}\label{eq:o19c}
\mathrm{dom}(\overline{\beta}_t) = \bigcup_{s\in \sigma^{-1}(t)} \mathrm{dom}(\beta_s) = \bigcup\limits_{s\in \sigma^{-1}(t)} D_{s^*}.
\end{equation}
The action of $\overline{\beta}_t$ on $\varphi\in {\mathrm{dom}}(\overline{\beta}_t)$ is given by
\begin{equation}\label{eq:o19d}
\overline{\beta_t}(\varphi)(e) = \varphi((es)^*),
\end{equation}
where $s$ is an arbitrary element of $\sigma^{-1}(t)$ is such that $\varphi\in D_{s^*}$.

\subsection{Proper restriction semigroups from partial monoid actions} In this section, we introduce a new construction of restriction semigroups from monoids and semilattices using partial actions. We start from the definition of a proper partial action.

\begin{definition} (Proper partial actions) Let $T$ be a monoid, $X$ a locally compact Hausdorff space and $E$ a  meet-subsemilattice of the generalized Boolean algebra ${\mathcal B}(X)$.  Let, further, $\alpha\colon T \to \PT(X)$ be a partial action. We say that $\alpha$ is {\em proper with respect to} $E$, if the following conditions hold:
\begin{enumerate}
\item[(P1)] for every $t\in T$: ${\mathrm{dom}}(\alpha_t)=\cup \{e\in E\colon e \subseteq {\mathrm{dom}}(\alpha_t)\}$;
\item[(P2)] for every $t\in T$, $e,f\in E$ where $e \subseteq {\mathrm{dom}}(\alpha_t)$: $\alpha_t^{-1}(f)\cap e \in E$.
\end{enumerate}
\end{definition}
 
Let $S$ be a restriction semigroup with local units. Recall from Proposition \ref{prop:embedding} that the map $\iota\colon S\to \mC(S)^a$, $s \mapsto (s, D_{s^*})$ is a $(\cdot\,,^*)$-embedding and, consequently, $S\simeq \iota(S)$.
 
\begin{proposition} \label{prop:spectral26} Let $S$ be a proper restriction semigroup with local units and write $E=P(S)$. Then the partial action $\overline{\beta}$ of $S/\sigma$ on $\widehat{E}$ is proper with respect to $\iota(E)$.
\end{proposition}
 
To prove this proposition, we first make the following observation.
 
\begin{lemma}\label{lem:u1} Suppose $t\in S/\sigma$ and $e\in E$ are such that $D_e\subseteq {\mathrm{dom}}(\overline{\beta}_t)$. Then there is $u\in \sigma^{-1}(t)$ such that $D_e = D_{u^*}$.
\end{lemma}
 
\begin{proof}
Since $\chi_{e^{\uparrow}} \in D_e$, \eqref{def:varphi1} implies that there is $s\in \sigma^{-1}(t)$ such that $\chi_{e^{\uparrow}} \in D_{s^*}$. It follows that $s^{*}\geq e$. Letting $u=se$, we have  $[u]_{\sigma} = t$, so that $D_{e} = D_{u^*}$, as desired.
\end{proof}
 
 \begin{proof}[Proof of Proposition \ref{prop:spectral26}.] Since ${\mathrm{dom}}(\beta_s) = D_{s^*}$, (P1) follows from \eqref{def:varphi1}. Let us turn to (P2). Let $e,f \in E$ and $t\in S/\sigma$ be such that $D_e\subseteq {\mathrm{dom}}(\overline{\beta}_t)$.  Lemma \ref{lem:u1} implies that $D_e = D_{u^*}$ for some $u\in \sigma^{-1}(t)$.
 Note that $(f,D_f), (u, D_{e})\in {\mathscr C}(S)^a$. Using \eqref{eq:subs_mult} and \eqref{eq:def_mult}, we have:
\begin{align*}
(f,D_f) (u, D_{e}) & =\{[f,x][u,y] \colon x\in D_f, y\in D_{e} \text{ and } \beta_u(y)=x\} \\ 
& = \{[fu,y] \colon y\in \beta_u^{-1}(D_f)\} \\
& = (fu, \beta_u^{-1}(D_f)).
\end{align*}
Since $\beta_u^{-1}(D_f) = \overline{\beta}_t^{-1}(D_f)\cap D_e$, the latter rewrites to $(fu, \overline{\beta}_t^{-1}(D_f)\cap D_e)$.
On the other hand, $(f,D_f) (u, D_{e}) = (fu, D_{(fu)^*})$, by \eqref{eq:prod25a}. Therefore, $\overline{\beta}_t^{-1}(D_f)\cap D_e = D_{(fu)^*}\in \iota(E)$, which completes the proof.
\end{proof}

Recall that a subset $I$ of a poset $P$ is called an {\em order ideal} if $e\in I$ and $f\leq e$ imply that $f\in I$. An order ideal is {\em principal} if it is of the form $e^{\downarrow}$ for some $e\in P$. We will now show that the open sets of $\widehat{E}$ of the form $\bigcup\limits_{Y\subseteq E} D_e$ are in a bijection with order ideals of $E$.
Let
$$
{\mathscr A} = \{\bigcup\limits_{e\in Y} D_e\colon Y\subseteq E\},
$$
and let ${\mathscr B}$ be the set of all order ideals of $E$. Define the map
$$\Psi\colon {\mathscr A} \to {\mathscr B}, \qquad
\bigcup_{e\in Y} D_e \mapsto \{f\in E\colon D_f\subseteq \bigcup_{e\in Y} D_e\}.
$$

\begin{proposition}\label{prop:oi}  The map $\Psi$ is a bijection with the inverse bijection 
$$\Psi^{-1} \colon {\mathscr B}\to {\mathscr A}, \qquad \Psi^{-1}(I) = \bigcup\limits_{e\in I} D_e.
$$
\end{proposition}

\begin{proof} Let the map $\psi \colon {\mathscr B}\to {\mathscr A}$ be given by $I \mapsto \bigcup\limits_{e\in I} D_e$. We show that $\Psi\psi$ is the identity map on ${\mathscr B}$ and $\psi\Psi$ is the identity map on ${\mathscr A}$. Let $I$ be an order ideal of $E$. We prove $I = \Psi\psi(I)$. By definition,
$$\Psi\psi(I) = \{f\in E\colon D_f\subseteq \bigcup\limits_{e\in I} D_e\}.$$
For the inclusion $\Psi\psi(I)\subseteq I$, let $f\in E$ be such that $D_f\subseteq \bigcup\limits_{e\in I} D_e$. Since the characteristic function $\chi_{f^{\uparrow}}$ belongs to $D_f$, there must be some $e\in I$ such that $\chi_{f^{\uparrow}} \in D_e$, which implies $\chi_{f^{\uparrow}}(e)=1$, so that $e\geq f$. Since $I$ is an order ideal and $e\in I$, we have $f\in I$, so that $\Psi\psi(I)\subseteq I$. The reverse inclusion, $I\subseteq \Psi\psi(I)$, holds trivially by the definitions of $\Psi$ and $\psi$, establishing the equality $I = \Psi\psi(I)$.

Now, let $A=\bigcup\limits_{e\in Y} D_e \in {\mathscr A}$. We prove that $A = \psi\Psi(A)$. We first characterize the elements of the ideal $\Psi(A)$. For any $f\in E$, we claim that $f\in \Psi(A)$ if and only if $f\leq e$ for some $e\in Y$. 

{\em Proof of Claim.} If $f\in \Psi(A)$, then $D_f\subseteq A$. Since $\chi_{f^{\uparrow}}\in D_f$, it follows that $\chi_{f^{\uparrow}} \in D_e$ for some $e\in Y$. This condition means $\chi_{f^{\uparrow}}(e)=1$, which implies $f\leq e$, as claimed. 

Using this characterization, the equality follows immediately: $$\psi(\Psi(A)) = \bigcup\limits_{e\in \Psi(A)}D_e = \bigcup\limits_{e\in Y} D_e = A.$$ This establishes that $\psi\Psi$ is the identity map on ${\mathscr A}$, completing the proof.
\end{proof}

The bijection $\Psi$ allows us to view the domains of the maps $\overline{\beta}_t$ for $t\in S/\sigma$ as order ideals of $E$. This connection will be further clarified  in Section~\ref{s:new_E_unit}.
We now show that this provides a convenient way to characterize $F$-restriction semigroups.

Recall that a restriction semigroup is called $F$-{\em restriction} if every its $\sigma$-class has a maximum element. Similarly to the proof of \cite[Lemma 5]{K15}, it can be shown that every $F$-restriction semigroup is automatically proper. We obtain the following statement, which parallels the known result for $F$-birestriction (and, in particular, $F$-inverse) monoids, see \cite[Lemma 6]{K15}. 

\begin{proposition} Let $S$ be a proper restriction semigroup with local units. Then $S$ is $F$-restriction if and only if $\Psi({\mathrm{dom}}(\overline{\beta}_t))$ is a principal order ideal for every $t\in S/\sigma$.
\end{proposition}

\begin{proof} Let $t\in S/\sigma$ and let $\tilde{t}$ be the maximum element of the $\sigma$-class $\sigma^{-1}(t)$. It is immediate by \eqref{def:varphi1} that ${\mathrm{dom}}(\overline{\beta}_t) = {\mathrm{dom}}(\beta_{\tilde{t}}) = D_{\tilde{t}}^{\,*}$. Hence $\Psi({\mathrm{dom}}(\overline{\beta}_t)) = \tilde{t}^{*}{}^{\downarrow}$.
Conversely, if $\Psi({\mathrm{dom}}(\overline{\beta}_t)) = e^{\downarrow}$ then ${\mathrm{dom}}(\overline{\beta}_t) = \Psi^{-1}(e^{\downarrow}) = \bigcup\limits_{f\leq e} D_f = D_e$.
Applying Lemma \ref{lem:u1}, we write $D_e=D_{s^*} = {\mathrm{dom}}(\beta_s)$ for some $s\in \sigma^{-1}(t)$. If $u\in \sigma^{-1}(t)$ then $u=su^* \leq ss^* = s$, so that $s$ is the maximum element of $\sigma^{-1}(t)$.
\end{proof}

We now show that proper partial actions can be used to construct proper restriction semigroups. 

\begin{proposition} \label{prop:restr} Let $T$ be a monoid, $X$ a locally compact Hausdorff space and $E$ a  meet-subsemilattice of the Boolean algebra ${\mathcal B}(X)$. Let, further, $\alpha$ be a proper partial action of $T$ on $X$ with respect to $E$. On the set 
\begin{equation}\label{eq:def_S}
S = \{(t, e) \colon t\in T, e\in E, e\subseteq {\mathrm{dom}}(\alpha_t)\}
\end{equation}
define the operations
\begin{equation}\label{eq:def_operations}
(s, e)(t, f) = (st, \alpha_t^{-1}(e)\cap f), \quad (s, e)^* = (1, e).
\end{equation}
Then:
\begin{enumerate} 
\item $(S; \, \cdot, \,^*)$ is a restriction semigroup. 
\item $P(S) = \{(1, e) \colon e\in E\}$
and the map $E\to P(S)$, $e\mapsto (1, e)$ is a semilattice isomorphism. 
\item The natural partial order $\leq$ on $S$ is given by $(s,e)\leq (t,f)$ if and only if $s=t$ and $e\subseteq f$.
\item 
The congruence $\sigma$ on $S$ is given by
$(t, e) \mathrel{\sigma} (s, f)$ if and only if $t=s$. Consequently, $S/\sigma \simeq T$ via the isomorphism $(t, e) \mapsto t$.
\item The restriction semigroup $S$ is proper.
\end{enumerate}
\end{proposition}

\begin{proof} (1) For the duration of this proof, we set $\mC=T\ltimes X$. Proposition \ref{prop:a21} implies that the sets $(t, e)$ are compact slices; therefore, $S\subseteq \mC^a$. Since $S$ is also closed with respect to the product and the operation $^*$, it is a $(\cdot\,,^*)$-subalgebra of $\mC^a$, meaning it is a restriction semigroup. 

(2) It is immediate by the definition of the $^*$ operation that projections of $S$ are precisely the elements $(1, e)$ where $e\in E$. The statement about the isomorphism with $E$ is also clear.

(3) By definition, we have that $(s,e)\leq (t,f)$ if and only if $(s,e) = (t,f)(s,e)^* = (t,f)(1,e) = (t, f\cap e)$. It follows that $s=t$ and $f\subseteq e$, as required. 

(4) Since $(t, {e\wedge f}) \leq (t, e), (t, f)$, we have $(t, e) \mathrel{\sigma} (t, f)$. Furthermore, a slice below $(t, e)$ must be contained in $(t, e)$ and is therefore of the form $(t, h)$ for some $h\leq e$. It follows that $(t, e) \mathrel{\sigma} (s, f)$ implies $s=t$, meaning that there is a bijection 
\begin{equation*}\label{eq:b21}
S/\sigma \to T, \quad (t, e) \mapsto t
\end{equation*}
between the set of $\sigma$-classes of $S$ and $T$.
Since, in addition, we have that $(s, e)(t, f) = (st,  \alpha_t^{-1}(e)\cap f)$, this bijection preserves the multiplication. It also preserves the $^*$ operation, since all the projections are mapped to $1$.

(5) Suppose that $(s, e),(t, f)\in S$ are such that $(t, e) \mathrel{\sigma} (s, f)$ and $(t, e)^* = (s, f)^*$. Parts (3) and (4) imply that $s=t$ and $e=f$, so that $(s, e)=(t, f)$. Hence $S$ is proper.
\end{proof}

A crucial special case of the construction of Proposition \ref{prop:restr} is obtained when $X$ is the spectrum $\widehat{E}$ of a semilattice $E$, yielding a restriction semigroup defined by a monoid $T$, a semilattice $E$ and a proper partial action of $T$ on  $\widehat{E}$. This leads to the following definition.

\begin{definition} (Partial action product of a semilattice by a monoid)
Let $T$ be a monoid and  $E$ a semilattice. Let, further, $\alpha$ be a proper partial action of $T$ on $\widehat{E}$ with respect to $\iota(E)$ (where $\iota$ is the embedding from Proposition \ref{prop:embedding}). The proper restriction semigroup $S$ constructed in Proposition \ref{prop:restr} will be  called the {\em partial action product} of $\iota(E)$ by $T$ {\em with respect to}  $\alpha$, and will be denoted by $T \ltimes_{\alpha} \iota(E)$.
\end{definition}

We now show that every proper restriction semigroup $S$ with local units is isomorphic to the partial action product of $\iota(P(S))$ by  $S/\sigma$ with respect to the partial action $\overline{\beta}$. 
 
\begin{theorem} (Structure of proper restriction semigroups with local units in terms of partial actions) \label{th:main1}
Let $S$ be a proper restriction semigroup with local units. Let $\beta$ be the spectral action of $S$ on $\widehat{P(S)}$. Then $S$ is $(\cdot\,,^*)$-isomorphic
to  $S/\sigma \ltimes_{\overline{\beta}} \iota(P(S))$ via the map 
\begin{align*} \label{eq:isom26}
\psi\colon & S\to S/\sigma \ltimes_{\overline{\beta}} \iota(P(S)), \\
& s\mapsto ([s]_{\sigma}, D_{s^*}).
\end{align*}
\end{theorem}
 
\begin{proof}  Since Proposition \ref{prop:spectral26} implies that $\overline{\beta}$ is a proper partial action with respect to $\iota(P(S))$, we can form the desired partial action product, $S/\sigma \ltimes_{\overline{\beta}} \iota(P(S))$. 
We first prove that $\psi$ is bijective. For injectivity, suppose $\psi(s)=\psi(t)$. This means $([s]_{\sigma}, D_{s^*}) = ([t]_{\sigma}, D_{t^*})$, which implies $s^*=t^*$ and $[s]_{\sigma} = [t]_{\sigma}$. Since $S$ is proper, we conclude that $t=s$. Hence, $\psi$ is injective.
 
For surjectivity,  let $(t, D_e) \in S/\sigma \ltimes_{\overline{\beta}} P(S)$. By the definition of the partial action product, $D_e\subseteq {\mathrm{dom}}(\overline{\beta}_t)$. Lemma \ref{lem:u1} ensures that  $D_e = D_{u^*}$ for some $u\in \sigma^{-1}(t)$. 
Thus, we can write $(t,D_e) = ([u]_{\sigma}, D_{u^*}) = \psi(u)$, implying that $\psi$ is surjective. 
 
We next show that $\psi$ is a $(\cdot\,,^*)$-morphism. By the definition of $\overline{\beta}_t$, we have that  $\overline{\beta}_{[t]_{\sigma}}(x) = \beta_t(x)$ for all $x\in D_{t^*}$. It follows that $\overline{\beta}_{[t]_{\sigma}}^{-1}(D_{s^*}) \cap D_{t^*} = {\beta}_{t}^{-1}(D_{s^*})$. Using this and \eqref{eq:def_operations}, we calculate:
\begin{align*}
\psi(s)\psi(t) & = ([s]_{\sigma}, D_{s^*})([t]_{\sigma}, D_{t^*}) \\
& = ([st]_{\sigma}, \overline{\beta}_{[t]_{\sigma}}^{\,-1}(D_{s^*}) \cap D_{t^*}) \\ 
& =  ([st]_{\sigma}, {\beta}_{t}^{-1}(D_{s^*})) \\
& =  ([st]_{\sigma}, D_{(st)^*}),
\end{align*}
$$
\psi(s)^* = ([s]_{\sigma}, D_{s^*})^* = (1, D_{s^*}) = \psi(s^*),
$$
as required. This completes the proof.
 \end{proof}
 Theorem \ref{th:main1} readily specializes to the structure of proper range, birestriction and inverse semigroups. We formulate the variant for proper inverse semigroups, commonly known as $E$-unitary inverse semigroups. We omit the details since they are similar to the restriction case.
 
\begin{corollary} (Structure of $E$-unitary inverse semigroups in terms of partial actions) \label{cor:E_unit}
Let $S$ be an $E$-unitary inverse semigroup. Then $S/\sigma \ltimes_{\overline{\beta}} \iota(E(S))$ is an inverse semigroup with 
\begin{align*}
(s, D_{e})(t, D_{f}) & = (st, \overline{\beta}_t^{-1}(D_e)\cap D_f), \\
(s, D_{s^*})^{-1} & = (s^{-1}, D_{s^+}),
\end{align*}
 where $s^* = s^{-1}s$ and $s^+=ss^{-1}$. Moreover, the map 
 $\psi\colon S\to S/\sigma \ltimes_{\overline{\beta}} \iota(E(S))$, given by 
$s\mapsto ([s]_{\sigma}, D_{s^*})$ is an isomorphism.
\end{corollary}

\subsection{A new proof of the structure theorem for $E$-unitary inverse semigroups}\label{s:new_E_unit}
In this section, we apply Corollary \ref{cor:E_unit} to provide a new proof for the structure result of $E$-unitary inverse semigroup due to Petrich and Reilly \cite[Theorem 3.2]{PR79}, see also \cite{KL04}. We begin by recalling the structure result from \cite{PR79}. 
 
Let $S$ be an $E$-unitary inverse semigroup. The {\em underlying partial action} of $S$ is a premorphism, $\psi$, from $S$ to the inverse semigroup $\Sigma(E(S))$ of all order-isomorphisms between order ideals of $S$. By definition, we have, for $s\in S$, that 
$$
{\mathrm{dom}}(\psi_s) = \{e\in E(S)\colon e\leq t^* \text{ for some } t\in \sigma^{-1}(s)\}.
$$
If $e\in {\mathrm{dom}}(\psi_s)$ then $\psi_s(e) = ses^{-1}$. 
Let $S/\sigma\ltimes_{\psi} E(S)$ be the set of all pairs $(s,e)$ where $s\in S/\sigma$ and $e\in {\mathrm{dom}}(\psi_s)$. Setting
$$
(s,e)(t,f) = (st, \psi_t^{-1}(\psi_t(f)\wedge e)), \quad (s,e)^{-1} = (s^{-1}, \psi_s(e))
$$
makes $S/\sigma \ltimes_{\psi} E(S)$ an inverse semigroup. The following result was proved by Petrich and Reilly in \cite[Theorem 3.2]{PR79}.

\begin{theorem}\label{th:pr} Let $S$ be an $E$-unitary inverse semigroup. Then  $s\mapsto ([s]_{\sigma}, s^*)$ is an isomorphism from $S$ onto $S/\sigma\ltimes_{\psi} E(S)$.
\end{theorem} 

For our new proof of the Petrich-Reilly result, we first establish the connection between $\overline{\beta}_s$ and $\psi_s$, where $s$ is an element of $S$.

\begin{lemma}\label{lem:o18a} Let $S$ be an $E$-unitary inverse semigroup, $s\in S$ and $e\in E(S)$.
Then: 
\begin{equation}\label{eq:o19a}
e\in {\mathrm{dom}}(\psi_s) \, \Leftrightarrow \, D_e\subseteq {\mathrm{dom}}(\overline{\beta}_s) \,\,\, \text{ and } \,\,\,
e\in {\mathrm{ran}}(\psi_s) \, \Leftrightarrow \, D_e\subseteq {\mathrm{ran}}(\overline{\beta}_s).
\end{equation}
Furthermore, 
\begin{equation}\label{eq:o19b}
\forall e\in {\mathrm{dom}}(\psi_s)\colon \overline{\beta}_s(D_e) = D_{\psi_s(e)} \,\, \text{ and } \,\, \forall e\in {\mathrm{ran}}(\psi_s)\colon \overline{\beta}_s^{-1}(D_e) = D_{\psi_s^{-1}(e)}.
\end{equation}
 \end{lemma}
\begin{proof} We start from proving the first equivalence in \eqref{eq:o19a}. Since $e\mapsto D_e$ is an isomorphism between $E(S)$ and $\iota(E(S))$, it is also an order-isomorphism, so that $e\leq f$ holds in $E(S)$ if and only $D_e\subseteq D_f$ holds in $\iota(E(S))$. In view of the definition of ${\mathrm{dom}}(\psi_s)$ and the definition $\overline{\beta}_s$ (see \eqref{eq:o19c}), we have:
$$
e\in {\mathrm{dom}}(\psi_s) \,\, \Leftrightarrow \,\, \exists u\in \sigma^{-1}(s) \colon e\leq u^* \,\, \Leftrightarrow \,\, \exists u\in \sigma^{-1}(s) \colon  D_e \subseteq D_{u^*} \,\, \Leftrightarrow \,\, D_e\subseteq {\mathrm{dom}}(\overline{\beta}_s).
$$
For the second equivalence, we first observe that ${\mathrm{ran}}(\psi_s) =  {\mathrm{dom}}(\psi_{s}^{-1})  = {\mathrm{dom}}(\psi_{s^{-1}})$ and ${\mathrm{ran}}(\overline{\beta}_s) = {\mathrm{dom}}(\overline{\beta}_{s}^{-1}) = {\mathrm{dom}}(\overline{\beta}_{s^{-1}})$. The second equivalence then follows directly from the first one, since $e\in {\mathrm{ran}}(\psi_s)$ if and only if $e\in {\mathrm{dom}}(\psi_{s}^{-1})$ and $D_e\subseteq {\mathrm{ran}}(\overline{\beta}_s)$ if and only if $D_e\subseteq {\mathrm{ran}}(\overline{\beta}_s^{-1})$.

Let now $e\in {\mathrm{dom}}(\psi_s)$. The equality $\overline{\beta}_s(D_e) = D_{\psi_s(e)}$ is a consequence of the following equivalences:
$$
\varphi \in \overline{\beta}_s(D_e)  \,\, \Leftrightarrow \,\,  \overline{\beta}_s(\varphi)(e) = 1  \,\, \Leftrightarrow \,\, \varphi((es)^*) = 1 \,\, \Leftrightarrow \,\,   \varphi\psi_s(e) = 1 \,\, \Leftrightarrow \,\, \varphi \in D_{\psi_s(e)}.
$$
This verifies the first formula in \eqref{eq:o19b}. The second formula follows by symmetry using the observation above.
\end{proof}

\begin{proposition}\label{prop:o18b} Let $S$ be an $E$-unitary inverse semigroup.
The map $$\gamma\colon S/\sigma\ltimes_{\psi} E(S) \quad \to \quad S/\sigma \ltimes_{\overline{\beta}} \iota(E(S)),$$ given by $(s, e) \mapsto (s, D_e)$ is an isomorphism. Consequently, 
$$S\to  S/\sigma\ltimes_{\psi} E(S), \quad s\mapsto ([s]_{\sigma}, s^*)$$ is an isomorphism.
\end{proposition}

\begin{proof} 
It is readily seen that the map $\gamma$ is a bijection. To show that it preserves the multiplication, we must show that 
$$
\gamma((st, \psi_t^{-1}(\psi_t(f)\wedge e))) = (s, D_e)(t, D_f) = (st, \overline{\beta}_t^{-1}(\overline{\beta}_t(D_f)\cap D_e)).
$$
It suffices to prove that $D_{\psi_t^{-1}(\psi_t(f)\wedge e)} = \overline{\beta}_t^{-1}(\overline{\beta}_t(D_f) \cap D_e)$. Applying Lemma \ref{lem:o18a}, we rewrite the left-hand side of this equality into
$\overline{\beta}_t^{-1}(D_{\psi_t(f)\wedge e})$.  Since $e\mapsto D_e$ is a morphism of semilattices, the latter rewrites to $\overline{\beta}_t^{-1}(D_{\psi_t(f)}\cap D_e))$. Therefore, we aim to prove that
\begin{equation}\label{eq:o19f}
\overline{\beta}_t^{-1}(D_{\psi_t(f)}\cap D_e)) = \overline{\beta}_t^{-1}(\overline{\beta}_t(D_f) \cap D_e).
\end{equation}
Applying $\overline{\beta}_t$ to both sides, we obtain the equivalent equality
$$
D_{\psi_t(f)}\cap D_e = \overline{\beta}_t(D_f) \cap D_e,
$$
which holds by \eqref{eq:o19b}. Therefore, \eqref{eq:o19f} also holds, which finishes the proof.
\end{proof}

Corollary \ref{cor:E_unit} and Proposition \ref{prop:o18b} bring a new proof of the Petrich-Reilly theorem (Theorem \ref{th:pr}) and consequently, a new proof of McAlister's $P$-theorem. The discussion in this section shows that Corollary \ref{cor:E_unit} can be viewed as the topological variant of Theorem~\ref{th:pr}.

\section{The algebra of an ample category and the isomorphism}\label{s:algebras}
\subsection{The convolution algebra of an ample category} The following definition is inspired by the definition of a convolution algebra of an ample groupoid \cite{St10}, known as a {\em Steinberg algebra}.

\begin{definition}\label{def:algebras} (Algebra of an ample category) Let $\mC$ be an ample category and $K$ a commutative unital ring. We define $K\mC$ to be the $K$-module spanned by the characteristic functions of elements of the Boolean restriction semigroup $\mC^a$.
\end{definition}

To align this definition with that of a Steinberg algebra, let $\mC$ be an ample groupoid. Its $K$-algebra, which we denote by $A$, is defined in \cite{St10} as the $K$-module spanned by the characteristic functions of compact bislices, that is, of elements of $\widetilde{\mC}^a$. Let us show that $A=K\mC$. To show that $K\mC\subseteq A$, we need to show that $\chi_U\in A$ for $U \in \mC^{a}$. Since compact bislices form a basis of the topology, $U$ is a finite union of compact bislices: $U=\cup_{i=1}^n U_i$.  Since  $U_i\subseteq U$ for all $i\in \{1,\dots, n\}$, the bislices $U_i$ are pairwise compatible. By Lemma \ref{lem:compatibility}, for any index set $J\subseteq \{1,\dots, n\}$, the intersection $\cap_{j\in J} U_j$ coincides with $U_{j_1} \prod_{j\in J} U_j^*\in \widetilde{\mC}^a$, where $j_1$ is an arbitrary element of $J$. Therefore, any finite intersection of elements of the set $\{U_1,\dots, U_n\}$ belongs to $\widetilde{\mC}^a$. Applying the inclusion-exclusion principle, we conclude that 
$$
\chi_U = \chi_{U_1\cup \dots \cup U_n} = \sum_{k=1}^n (-1)^{k-1} \sum_{J\subseteq \{1,\dots, n\}, |J|=k} \chi_{\cap_{j\in J}U_j},
$$
which means that $\chi_U\in A$, as required. Therefore, $K\mC\subseteq A$. Since any compact bislice is a compact slice, the reverse inclusion is obvious. It follows that $A=K\mC$, so that the algebra of $\mC$ defined in \cite{St10} coincides with the algebra $K\mC$ defined above. Therefore, the algebra $K\mC$ indeed generalizes the notion of a Steinberg algebra of an ample groupoid.

\begin{example} If $\mC$ is endowed with the discrete topology, $K\mC$ is the module of all functions on $\mC$ with finite support. It is spanned by the set $\{\delta_x\colon  x\in \mC\}$, where $\delta_x$ is the Kronecker delta:
$$
\delta_x(y) = \left\lbrace\begin{array}{ll} 1, & \text{if } \, x=y;\\ 0, & \text{otherwise.}\end{array}\right.
$$
\end{example}

For every $f,g\in K\mC$, we define their {\em convolution product} by
\begin{equation}\label{def:conv}
(f*g)(x) = \sum_{uv=x} f(u)g(v).
\end{equation}

\begin{lemma} \label{lem:mult1} If $U,V\in \mC^a$ then $\chi_U * \chi_V$ is well defined and equals $\chi_{UV}$. Consequently, the product in \eqref{def:conv} is well defined.
\end{lemma}

\begin{proof} Note that $\chi_U(u)\chi_V(v)=0$ unless $u\in U$ and $v\in V$ in which case $\chi_U(u)\chi_V(v)=1$. 
Observe that every $x\in UV$ admits a unique factorization $x=uv$, where $u\in U$ and $v\in V$. Indeed, if $x=st$ is another such a factorization then
we have $\dom(v) = \dom(t) = \dom(x)$. But $v,t\in V$ and $V$ is a slice. Hence $v=t$. It follows that $\dom(u) = \ran(v) = \ran(t) = \dom (s)$. Because $u,s\in U$ and $U$ is a slice, we now conclude that $u=s$, as desired. It follows that for every $x\in UV$ the sum $\sum_{uv=x} \chi_U(u)\chi_V(v)$ has precisely one non-zero term which equals $1$. On the other hand, when $x\not\in UV$, this sum equals $0$. The statement follows.
\end{proof}

\begin{proposition} The convolution product on $K\mC$ is associative. Consequently $K\mC$ is a $K$-algebra.
\end{proposition}

\begin{proof} Let $f,g,h\in K\mC$. Then
$$
((f*g)*h)(x) = \sum_{uv=x}(f*g)(u)h(v) = \sum_{uv=x}\sum_{pq=u} f(p)g(q)h(v) = \sum_{pqv=u} f(p)g(q)h(v)
$$
and, similarly, $(f*(g*h))(x) = \sum\limits_{pqv=u} f(p)g(q)h(v)$.
\end{proof}

\subsection{The isomorphism of algebras}
Let $K$ be a commutative unital ring and $S$ a semigroup. Recall that the {\em semigroup algebra} $KS$ is the free $K$-module with basis $S$ which is equipped with the product
$$
\sum_{s\in S}c_s s \cdot \sum_{t\in S} d_tt = \sum_{s,t\in S} c_sd_t st.
$$

Let $S$ be a restriction semigroup with local units and $\beta$ its spectral action, see Example~\ref{ex:spectral}. Recall that
$(s, D_{s^*}) = \{[s,\varphi]\colon \varphi \in D_{s^*}\}$. From Proposition \ref{prop:embedding} we know that $\iota\colon S\to {\mathscr C}(S)^a$, $s\mapsto (s, D_{s^*})$ is an embedding of restriction semigroups.
We now state the algebra isomorphism theorem.

\begin{theorem}\label{th:isom}
Let $S$ be a restriction semigroup with local units and $K$ a commutative unital ring. The map 
$$F\colon S\to K{\mathscr C}(S), \quad  s\mapsto \chi_{(s, D_{s^*})}$$
 extends to an isomorphism of algebras
$F\colon KS \to K{\mathscr C}(S)$.
\end{theorem}

\begin{proof} Let us show that $F$ extends to a homomorphism $F\colon KS \to  K{\mathscr C}(S)$. By Proposition
\ref{prop:embedding} we know that $(st, D_{(st)^*}) = (s, D_{s^*})(t, D_{s^*})$, which, in view of Lemma \ref{lem:mult1}, yields $\chi_{(st, D_{(st)^*})} = \chi_{(s, D_{s^*})}\chi_{(t, D_{t^*})}$, as desired.
Since all $(e, D_e)$, where $e$ runs through $P(S)$, are in the image of $F$ and generate the generalized Boolean algebra ${\mathsf{B}}(P(S))$, it follows that each $\chi_{(e, U)}$, where $U\in  {\mathsf{B}}(P(S))$ and $U\subseteq D_e$, is also in the image of $F$. Furthermore, for a compact slice $(s, U)$, where $U$ is a compact-open subset of $D_{s^*}$, the equality $(s, U)=(s, D_{s^*})(s^*, U)$ shows that $\chi_{(s, U)}$ is also in the image of $F$. Therefore, $F$ is surjective.

To show that $F$ is injective,
it suffices to prove that the set $\{\chi_{(s, D_{s^*})}\colon s\in S\}$ is linearly independent. Seeking a contradiction, assume that this is not so. Then there is a finite set $X=\{s_1,\dots, s_n\}\subseteq S$ such that 
\begin{equation}\label{eq:a}
\sum_{s\in X} \alpha_s \chi_{(s, D_{s^*})} =0,
\end{equation}
where all coefficients $\alpha_s\in K$ are non-zero. 

Let us fix an arbitrary $s\in X$. We evaluate the linear combination at the germ $[s, \chi_{(s^*)^{\uparrow}}]$. Since $[s, \chi_{(s^*)^{\uparrow}}] \in (s, D_{s^*})$, it follows that $\chi_{(s, D_{s^*})} ([s, \chi_{(s^*)^{\uparrow}}]) =1$. Consequently, the $s$-term 
contributes $\alpha_s\cdot 1 = \alpha_s\neq 0$ to the sum.  Since the total sum in \eqref{eq:a} equals zero, there must be some $u\in X \setminus \{s\}$ such that $\chi_{(u, D_{u^*})}([s, \chi_{(s^*)^{\uparrow}}]) =1$.
This means that $[s, \chi_{(s^*)^{\uparrow}}] \in (u, D_{u^*})$, which implies the equality $[s, \chi_{(s^*)^{\uparrow}}] = [u, \psi]$ for some $\psi \in D_{u^*}$. By the definition of the equality of germs, this implies that $\psi = \chi_{(s^*)^{\uparrow}} \in D_{u^*}$ and there is $v\leq s,u$ satisfying $\chi_{(s^*)^{\uparrow}}(v^*) =1$ or, equivalently, $v^*\geq s^*$.
Since $v\leq s$, we also have $v^*\leq s^*$. Therefore, $v^* = s^*$, which in turn yields $v=sv^* = ss^* = s$. Since $v\leq u$, we conclude that $s\leq u$. However, as we chose $u\neq s$, the inequality must be strict: $s\lneq u$. We have thus shown that for every $s\in X$ there is $u = u(s) \in X \setminus \{s\}$ such that $s\lneq u(s)$.
It follows that $u(s) \lneq u(u(s)) = u^2(s)$, etc. We thus construct a strictly increasing chain of elements of $X$:
$$
s\lneq u(s) \lneq u^2(s) \lneq u^3(s) \lneq  \dots
$$ 
which is impossible, since $X$ is finite.
The obtained contradiction proves that the set  $\{\chi_{(s, D_{s^*})}\colon s\in S\}$ is linearly independent, as required.
\end{proof}

Theorem \ref{th:isom} is analogous to Steinberg's result on algebras of inverse semigroups \cite[Theorem 6.3]{St10}. Our proof, applied to an inverse semigroup $S$, yields a direct proof of  \cite[Theorem 6.3]{St10} which, unlike the original proof in \cite{St10}, does not involve the universal groupoid of $S$ or M\"obius inversion.

\subsection{The finitary case} Let us look at the special case where $S$ satisfies the condition that for all $e\in P(S)$ the principal downset $e^{\downarrow}$ is finite.
Proposition \ref{prop:finitary} implies that
$$(s, D_{s^*}) = \{[se, ((se)^*)^{\uparrow}] \colon e \in P(S), e\leq s^*\} = \{[t, (t^*)^{\uparrow}]\colon t\leq s\}$$ is a finite set. It is readily verified that the topology on ${\mathscr C}(S)$ is discrete (cf. \cite[Proposition 2.5]{St10}), so that compact slices are then precisely finite slices. It follows that the functions $\delta_{[s,(s^*)^{\uparrow}]}$, where $s$ runs through $S$, form a basis of the algebra $K{\mathscr C}(S)$. Furthermore, we have:
$$
\chi_{(s, D_{s^*})} = \sum_{t\leq s} \delta_{[t, (t^*)^{\uparrow}]}.
$$
In the case where $S$ is a range semigroup\footnote{The isomorphism result of \cite[Theorem 1.5]{S18} is proved for range semigroups satisfying the condition that every principal downset of $P(S)$ is finite. We remark that range semigroups are referred to as $E$-Ehresmann left restriction semigroups in \cite{S18}.} this recovers the isomorphism from \cite[Theorem 1.5]{S18} (note that the element $\delta_{[s,(s^*)^{\uparrow}]}$ is denoted by $C(s)$ in \cite{S17,S18}). Therefore, Theorem \ref{th:isom} generalizes the isomorphism of \cite{S17,S18} from range semigroups, for which all principal downsets of $P(S)$ are finite, to arbitrary restriction semigroups with local units.

\section{Concluding remarks}\label{s:last}
In this final section, we outline some areas for future exploration suggested by the work done in the present paper. 

Since the connection between restriction semigroups and \'etale categories closely resembles the connection between inverse semigroups and \'etale groupoids, it is reasonable to seek extensions of the properties of \'etale groupoids, such as being  effective, principal, minimal, etc. to the broader setting of \'etale categories and connecting these properties with algebraic properties of the associated Boolean restriction semigroups of compact slices, extending the existing programme for Boolean inverse semigroups (see \cite{L23}). Another direction is to relate restriction semigroups and \'etale categories with topological binary relations (see \cite{Power92, HPP05}), which have existed for a while and are useful, for example, for classification of triangular subalgebras of certain groupoid $C^*$-algebras.

It also seems reasonable to look for further applications of the partial action product construction of  Section \ref{s:proper} and its variations. For example, we anticipate that the technique of Section \ref{s:proper} can be suitably adapted to extend the results on proper extensions of birestriction semigroups of \cite{DKK21} to describe proper extensions of restriction semigroups.

Our work suggests developing the theory of algebras and operator algebras of \'etale categories and restriction semigroups, generalizing the theory of Steinberg algebras and $C^*$-algebras of inverse semigroups and \'etale groupoids. An obvious task is to define operator algebras of a restriction semigroup with local units, as well as operator algebras of  \'etale categories, and prove analogues of Theorem \ref{th:isom}. 
Another direction is to extend the results of \cite{MC24} (established for the birestriction and bi\'etale case) to algebras of restriction semigroups and \'etale categories.  It would be also worthwhile to establish connections between  properties of the algebra $K{\mathscr C}(S)$ for a restriction semigroup with local units $S$, and properties of $S$ and ${\mathscr C}(S)^a$, thereby extending the existing programme for inverse semigroup algebras, see, e.g., \cite{StSz21}. 

Finally, an important research direction is seeking categorical models for non self-adjoint operator algebras of dynamical origin considered in the literature, see, e.g., \cite{Peters84, KK12} as well as of seeking new natural examples of algebras (in both analytical and algebraic settings) arising from \'etale categories.

\footnotesize
\def\bibspacing{-3pt}


\end{document}